\documentclass[11pt,a4paper]{amsart}
\usepackage[utf8]{inputenc}
\usepackage[english]{babel}
\usepackage{amsmath}
\usepackage{amsfonts}
\usepackage{amssymb}
\usepackage{amsthm}

\usepackage[top=3.5cm,bottom=3.5cm,left=3cm,right=3cm]{geometry}

\usepackage{hyperref}
\usepackage{tikz-cd}
\usepackage{stmaryrd}
\usepackage{enumitem}
\usepackage{comment}
\usepackage[normalem]{ulem}

\newtheorem {theorem}{Theorem}
\newtheorem* {theorem*}{Theorem}
\newtheorem {corollary}[theorem]{Corollary}
\newtheorem {lemma}[theorem]{Lemma}
\newtheorem {proposition}[theorem]{Proposition}

\theoremstyle{definition}
\newtheorem {remark}[theorem]{Remark}
\newtheorem {definition}[theorem]{Definition}
\newtheorem {example}[theorem]{Example}

\numberwithin{theorem}{section}

\DeclareMathOperator{\Aut}{Aut}
\DeclareMathOperator{\End}{End}
\DeclareMathOperator{\Gal}{Gal}
\DeclareMathOperator{\Hom}{Hom}
\DeclareMathOperator{\Emb}{Emb}

\DeclareMathOperator{\Mat}{Mat}

\DeclareMathOperator{\id}{id}

\DeclareMathOperator{\tors}{tors}

\DeclareMathOperator{\im}{Im}

\DeclareMathOperator{\rk}{rk}

\newcommand{\Q}{\mathbb{Q}}

\newcommand{\NN}{\mathbb{N}}
\newcommand{\Z}{\mathbb{Z}}

\newcommand{\Kbar}{\overline{K}}

\newcommand{\into}{\hookrightarrow}

\newcommand{\isomto}{\overset{\sim}{\to}}

\newcommand{\set}[1]{\left\{ #1 \right\}}
\newcommand\restr[2]{\ensuremath{\left.#1\right|_{#2}}}

\renewcommand{\geq}{\geqslant}

\renewcommand{\hat}{\widehat}

\newcommand{\extension}[1]{\texorpdfstring{$#1$}{#1}-extension}
\newcommand{\mapaux}[1]{\texorpdfstring{$#1$}{#1}-map}
\newcommand{\extensions}[1]{\texorpdfstring{$#1$}{#1}-extensions}
\newcommand{\mapauxs}[1]{\texorpdfstring{$#1$}{#1}-maps}
\newcommand{\hull}[1]{\texorpdfstring{$#1$}{#1}-hull}
\newcommand{\hulls}[1]{\texorpdfstring{$#1$}{#1}-hulls}
\newcommand{\jinj}{\texorpdfstring{$J$}{J}-injective}
\newcommand{\jextension}{\extension{J}}
\newcommand{\jextensions}{\extensions{J}}
\newcommand{\jmap}{\mapaux{J}}
\newcommand{\jmaps}{\mapauxs{J}}
\newcommand{\jtextension}{\extension{(J,T)}}
\newcommand{\jtextensions}{\extensions{(J,T)}}
\newcommand{\jhull}{\hull{J}}
\newcommand{\jhulls}{\hulls{J}}
\newcommand{\pointed}{\texorpdfstring{$T$}{T}-pointed}
\newcommand{\push}{saturation}
\newcommand{\tp}{\mathfrak{sat}}
\newcommand{\jt}{\mathfrak{tor}}
\newcommand{\jtcat}{\mathfrak{JT}}
\newcommand{\inct}[1]{\mathfrak{t}_{#1}}
\newcommand{\incs}[1]{\mathfrak{s}_{#1}}

\newcommand{\divm}[3]{\ensuremath{\left(#2:_{#3}#1\right)}}

\begin{document}

\title{Division in modules and Kummer Theory}
\author{Sebastiano Tronto}
\address[]{Department of Mathematics, University of Luxembourg,
6 av.\@ de la Fonte, 4364 Esch-sur-Alzette, Luxembourg}
\email{sebastiano.tronto@uni.lu}

\begin{abstract}
In this work we generalize the concept of injective module and develop
a theory of divisibility for modules over a general ring, which provides a
general and unified framework to study Kummer-like field extensions arising
from commutative algebraic groups. With these tools we provide an effective
bound for the degree of the field extensions arising from division points of
elliptic curves, extending previous results of Javan Peykar for CM curves
and of Lombardo and the author for the non-CM case.

\end{abstract}

\maketitle

\section{Introduction}
\label{section:intro}

Let $K$ be a number field and fix an algebraic closure $\Kbar$ of $K$.
If $G$ is a commutative connected algebraic group over $K$ and $A$ is
a subgroup of $G(K)$, we may consider for every positive integer $n$
the field extension $K(n^{-1}A)$ of $K$ inside $\Kbar$ generated by all points
$P\in G(\Kbar)$ such that $nP\in A$. This is a Galois extension of $K$
containing the $n$-torsion field $K(G[n])$ of $G$.

If $G=\mathbb{G}_m$ is the multiplicative group, extensions of this kind
are studied by classical Kummer theory. Explicit results for this case
can be found for example in
\cite{ps1}, \cite{pst} and \cite{pst2}.
The more general case of an extension of
an abelian variety by a torus is treated in Ribet's foundational paper \cite{ribet}.
Under certain assumptions, for example if $G$ is the product of an abelian
variety and a torus and $A$ is free of rank $r$ with a basis of points linearly
independent over $\End_K(G)$, it is known that the ratio
\begin{align}
	\label{eqn:firstRatio}
	\frac{n^{rs}}{\left[K\left(n^{-1}A\right):K(G[n])\right]}
\end{align}
where $s$ is the unique positive integer such that $G(\Kbar)[n]\cong (\Z/n\Z)^s$
for all $n\geq 1$, is bounded independently of $n$
(see also \cite[Th\'eor\`eme 5.2]{bertrand} and \cite[Lemme 14]{hindry}).

In the case of elliptic curves, one may hope to obtain an explicit version of
this result. Indeed the results of \cite{lt} and \cite{mypaper} provide such
a statement under the assumption that $\End_K(G)=\Z$, and they show that an
effective bound depends only on the abelian group structure of $A$ and on the
$\ell$-adic Galois representations associated with the torsion of $G$ for
every prime $\ell$.

It is clear from the discussion above that the existence of non-trivial
endomorphisms defined over $K$ plays an essential role in this theory.
Without loss of generality we can take $A$ to be an $\End_K(G)$-module,
as done by Javan Peykar in his thesis \cite{abtien}.
This approach leads to an
explicit ``open image theorem'' for Kummer extensions for CM elliptic curves,
albeit under certain technical assuptions on $\End_K(G)$.

Motivated by \cite{abtien} and by the author's previous results \cite{mypaper},
most of this paper is devoted to developing a general abstract framework for
the study of certain \emph{division modules} of a fixed $R$-module $M$, where
$R$ is any unitary ring. We strive to develop this theory in a way that is
independent from the ``ambient module'' $G(\Kbar)$, taking inspiration
from \cite{pmaster} as well.

We introduce a natural generalization of the concept of injective modules,
which to the author's knowledge is novel. We also define a category of
\emph{\jtextensions{}}, which shares many interesting properties with the
category of field extensions.
We believe that these topics are interesting in their own right.

At the end of the paper we prove the following result, which was previously
known in this effective form only under certain restrictions on $\End_K(E)$:

\begin{theorem*}
	Let $E$ be an elliptic curve over a number field $K$, let
	$R=\End_K(E)$ and let $M$ be an $R$-submodule of $E(K)$.
	There exists a positive integer $c$, depending only on the
	$R$-module structure of $M$ and on the image of the
	Galois representations associated with the torsion of $E$,
	such that for every positive integer $n$
	\begin{align*}
		\frac{n^{2\rk_R(M)}}{[K(n^{-1}M):K(E[n])]} \qquad
		\text{divides} \qquad c\,.
	\end{align*}
\end{theorem*}

This result follows from Theorem \ref{thm:Main}, which is essentially an application
of Theorem \ref{thm:main}, which in turn is a generalization of
\cite[Theorem 5.9]{mypaper}. The results on Galois representations needed to
apply this general theorem are mostly taken from \cite{lt}, and it can be
easily seen that the given bounds only depend on the $\ell$-adic representations,
so that the constant $c$ of our main theorem is effectively computable.

\subsection{Notation}

In this paper, rings are assumed to be unitary, but not necessarily
commutative; subrings always contain the multiplicative unit $1$.
Unless otherwise specified, by ideal of a ring we mean a right ideal and
by module over a ring we mean a left module. If $R$ is a ring and $n$ is
a positive integer, we will denote by $\Mat_{n\times n}(R)$ the ring
of $n\times n$ matrices with coefficients in $R$.

We denote by $\Z$ the integers and by $\Z_{>0}$ the set of positive integers.
If $p$ is a prime number we denote by $\Z_p$ the completion of the ring $\Z$
at the ideal $(p)$. We denote by $\hat \Z$ the product of $\Z_p$ over all primes
$p$, which we identify with $\varprojlim_{n\in\Z_{>0}}\Z/n\Z$.

\subsection{Structure of the paper}

In the Section \ref{section:jinj}
we introduce the concept of \emph{ideal filter} and of
division module by an ideal filter. This provides us with a way to generalize
the notion of injective module, and we are able to show the equivalent of
Baer's criterion for injectivity and the existence of the analogue of
injective hulls in this setting.
At the end of Section \ref{section:jinj}
we prove a certain duality result for $J$-injective
modules that will be applied in Section \ref{section:kummer}.

In Section \ref{section:jtext}
we construct the category of \jtextensions{}, our abstraction
for the modules of division points of an algebraic group. This category behaves
similarly to that of field extensions of a given field. After studying an
interesting pair of adjoint functors, we conclude this section by proving
the existence of a \emph{maximal} \jtextension{}, in analogy with field theory.

Section \ref{section:aut} is devoted to the study of automorphism groups of \jtextensions{}.
The fundamental exact sequence of Theorem \ref{thm:exactSequence} gives us a
framework to study the Galois groups of Kummer extensions associated
with a commutative algebraic group, provided that some technical assumptions
hold. This is what we do in Section \ref{section:kummer},
and we conclude by applying these results
to elliptic curves.

\subsection*{Acknowledgements}
I would like to thank my supervisors Antonella Perucca and Peter Bruin for
their constant support. I would also like to thank Hendrik Lenstra and Peter
Stevenhagen for the interesting discussion about the results of \cite{abtien}
which gave me the main ideas for this paper.
Last but not least, I would like to thank Davide Lombardo for his comments
on this paper, in particular for suggesting Remarks \ref{remark:maximalorders}
and \ref{remark:av}.


\section{\texorpdfstring{$J$}{J}-injectivity}
\label{section:jinj}

\subsection{Ideal filters and division in modules}

In order to study division in modules over a general ring, we take inspiration
from \cite{abtien}. However, instead of using Steinitz ideals (that is,
ideals of the completion of a ring), we use a more general concept that we
now introduce.

\begin{definition}
	Let $R$ be a ring. We call a non-empty set $J$ of right
	ideals of $R$ an \emph{ideal filter} if the following conditions hold:
	\begin{enumerate}
		\item If $I,I'\in J$ then $I\cap I'\in J$, and
		\item If $I\in J$ and $I'$ is a right ideal of $R$ containing
		      $I$, then $I'\in J$.
	\end{enumerate}
\end{definition}

The minimal ideal filter is $\{R\}$, while the maximal ideal filter contains
all ideals (equivalently, it contains the zero ideal): we denote the former
by $1$ and the latter by $0$.

For any ring $R$ and any set $S$ of right ideals of $R$ we call the
ideal filter \emph{generated} by $S$ the smallest ideal filter containing $S$:
it consists of all ideals of $R$ which contain a finite intersection of
elements of $S$.

\begin{example} \label{example:pInfty}
	We will be interested in the ideal filters generated by the powers
	of a given prime number $p$
	\[ p^\infty := \set{ I\text{ right ideal of } R \mid I\supseteq p^nR 
		\text{ for some } n\in\NN} \]
	and the one generated by all non-zero integers
	\[ \infty := \set{ I \text{ right ideal of } R \mid I\supseteq nR
		\text{ for some } n\in \Z_{>0}} \, .\]
	Notice that if $p^n=0$ (resp $n=0$) for some $n\in\Z_{>0}$
	then $p^\infty$ (resp. $\infty$) is
	simply the maximal ideal filter $0$. We will often consider
	such ideal filters in the case where $R$ is a commutative integral
	domain of characteristic different from $p$ (resp. characteristic $0$).
\end{example}

Fix for the remainder of this section a ring $R$.

\begin{definition}
	If $M\subseteq N$ are left $R$-modules, for any right ideal
	$I$ of $R$ we call 
	\begin{align*}
		\divm IMN :=\set{x\in N\mid Ix\subseteq M}
	\end{align*}
	the \emph{$I$-division module of $M$ in $N$}.
\end{definition}

A similar concept for ideals of $R$ is sometimes referred to as
\emph{quotient ideal}, but we deemed appropriate a change of terminology.

We can easily generalize this notion to ideal filters of $R$.

\begin{definition}
	Let $J$ be an ideal filter of $R$ and let $M\subseteq N$ be left
	$R$-modules.
	We call
	\begin{align*}
		\divm{J}{M}{N} &:= \bigcup_{I\in J} \divm{I}{M}{N}
	\end{align*}
	the
	\emph{$J$-division module of $M$ in $N$}.
	One can easily check that
	$\divm JMN$ is an $R$-submodule of $N$. 

	Moreover, we call $N[J]:=\divm J0N$
	the \emph{$J$-torsion submodule} of $N$.
	We call $N$ a \emph{$J$-torsion module} if $N=N[J]$.
\end{definition}

\begin{remark}
	If $J=0$ then $\divm JMN=N$ and $M[J]=M$. On the other hand, if
	$J=1$ then $\divm JMN = M$ and $M[J]=0$.
\end{remark}

\begin{remark} \label{remark:jdivContained}
	Let $M\subseteq N$ be left $R$-modules and let $J$ and $J'$ be ideal
	filters of $R$ with $J'\subseteq J$. If $M'\subseteq M$ and $N'\subseteq N$ 
	are submodules with $M'\subseteq N'$ then it is clear from the definition 
	of $J$-division module that $\divm{J'}{M'}{N'}\subseteq \divm JMN$.
\end{remark}

\begin{definition}
	We say that an ideal filter $J$ of $R$ is \emph{complete} if for every
	left $R$-module $N$ and every submodule $M\subseteq N$ we have
	\[ \divm J{\divm JMN}N = \divm JMN \,.\]

	We say that an ideal filter $J$ is \emph{product-closed} if for any
	$I,I'\in J$ we have $II'\in J$.
\end{definition}

\begin{proposition}
	Let $R$ be a ring and let $J$ be a product-closed ideal filter
	of $R$. If for every $I\in J$ the left ideal $RI$ is finitely generated,
	then $J$ is complete. In particular, every product-closed ideal filter
	over a left-Noetherian ring is complete.
\end{proposition}
\begin{proof}
	Let $J$ be a product-closed ideal filter of $R$ and let $M\subseteq N$ be
	left $R$-modules. The inclusion $\divm JMN\subseteq \divm J{\divm JMN}N$
	is always true, so let us prove the other inclusion.
	Let $x\in N$ be such that there is $I\in J$ with
	$Ix\subseteq \divm JMN$. Let $\{y_1,\dots y_n\}$ be a set of generators
	for the left ideal $RI$. Then for every $i=1,\dots n$ there is $I_i\in J$
	such that $I_iy_ix\subseteq M$. By definition of ideal filter we have
	$I':=\bigcap_{i=1}^nI_i\in J$ and since $J$ is product-closed we have
	$I'I\in J$. Since $\{y_1,\dots, y_n\}$ is a set of generators of the
	left ideal $RI$ and $I'$ is a right ideal we have
	$I'Ix=I'(RI)x\subseteq M$, which shows that $J$ is complete.
\end{proof}

\begin{example}
	The ideal filters introduced in Example \ref{example:pInfty} are both
	product-closed. If, for example, $R$ is Noetherian, then they are
	also complete.
\end{example}

We conclude this subsection with a list of properties of division modules.

\begin{lemma} \label{lemma:propertiesOfDivision}
	Let $M\subseteq N\subseteq P$ and $M'$ be left $R$-modules and let $J$ and
	$J'$ be ideal filters of $R$. Then the following properties hold:
	\begin{enumerate}
		\item $\divm JMN = \divm JMP\cap N$.
		\item $\divm{J}{M}{\divm JMN} =\divm JMN$.
		\item $(N/M)[J]=\divm JMN/M$.
		\item $\divm JMN=N$ if and only if $N/M$ is $J$-torsion.
		\item $(M\oplus M')[J]=M[J]\oplus M'[J]$.
	\end{enumerate}
\end{lemma}

\begin{proof}\leavevmode
	\begin{enumerate}
		\item The inclusion $\divm JMN \subseteq \divm JMP\cap N$ is obvious;
		      for the other inclusion it suffices to notice that if $n\in N$ is
		      such that $In\subseteq M$ for some $I\in J$ then by definition
		      $n\in\divm JMN$.
		\item Follows directly from (1).
		\item We have
			\begin{align*}
				(N/M)[J]&=\bigcup_{I\in J}(N/M) [I]=\\
						&=\bigcup_{I\in J} \set{n+M\in N/M\mid I(n+M)=M}=\\
						&=\bigcup_{I\in J} \set{n\in N\mid In\subseteq M}/M=\\
						&=\bigcup_{I\in J} \divm IMN /M=\\
						&=\divm JMN/M\,.
			\end{align*}
		\item By (3) we have that $(N/M)[J]=N/M$ if and only if $\divm JMN = N$.
		\item For any right ideal $I$ of $R$ and any $(m,m')\in M\oplus M'$
			we have that $I(m,m')=0$ if and only if $Im=Im'=0$. This implies
			that $(M\oplus M')[I]= M[I]\oplus M'[I]$, so we have
			\begin{align*}
				(M\oplus M')[J]&= \bigcup_{I\in J}(M\oplus M') [I]=\\
				&=\bigcup_{I\in J}M[I]\oplus M'[I]=\\
				&=M[J]\oplus M'[J].
			\end{align*}
	\end{enumerate}
\end{proof}

\subsection{\jmaps{} and \jextensions{}}

Fix for this section a ring $R$ and a complete ideal filter $J$
of $R$.
We introduce here some simple notions that will lead us closer to our
definition of \jtextensions{}.

\begin{definition}
	Let $M$ be a left $R$-module. An $R$-module homomorphism 
	$\varphi:M\to N$ is called a \emph{\jmap{}} if $\divm{J}{\varphi(M)}{N}=N$.
	If $\varphi$ is injective we will call it a \emph{\jextension{}},
	and we say that $N$ is a \jextension{} of $M$.
\end{definition}

\begin{remark}
	By Lemma \ref{lemma:propertiesOfDivision}(4) a homomorphism
	$\varphi:M\to N$ is a \jmap{} if and only if $N/\varphi(M)$ is $J$-torsion.
	In particular, if $J=0$ then every homomorphism of $R$-modules is
	a $J$-map.
\end{remark}

It is clear from the definition that if $\varphi:M\to N$ and $\psi:M\to P$ are
two \jmaps{} then any $R$-module homomorphism $f:N\to P$ such that
$f\circ\varphi=\psi$ is also a \jmap{}.

The following Lemma, which strongly relies on the assumption that $J$ is
complete, shows moreover that $R$-modules and \jmaps{} form a
subcategory of the category of $R$-modules.

\begin{lemma} \label{lemma:transAndGap}
	Let $M,N$ and $P$ be $R$-modules and let $\varphi:M\to N$ and
	$\psi:N\to P$ be $R$-module homomorphisms. If $\varphi$ and $\psi$ are
	\jmaps{}, then so is $\psi\circ \varphi$.
\end{lemma}

\begin{proof}
	Since $J$ is complete we have
	\begin{align*}
		P&=\divm J{\psi(N)}P=\\
		 &=\divm J{\divm J{\psi\varphi(M)}{\psi(N)}}P\subseteq \\
		 &\subseteq \divm J{\divm J{\psi\varphi(M)}P}P= \\
		 & =\divm J{\psi\varphi(M)}P
	\end{align*}
	hence $\divm J{\psi\varphi(M)}P=P$ and $\psi\circ\varphi$ is a \jmap{}.
\end{proof}

\begin{remark} \label{rem:jmapjtorsion}
	Any homomorphism of $R$-modules $\varphi:M\to N$ such that $N$ is
	$J$-torsion is a \jmap{}. In particular, the restriction of an $R$-module
	homomorphism to the $J$-torsion submodule is a \jmap{}.
\end{remark}

The following Lemma illustrates how certain properties of a \jmap{}
largely depend on its restriction to the $J$-torsion submodule.
Recall that an injective $R$-module homomorphism $f:M\into N$ is called
an \emph{essential extension} if for every submodule $N'\subseteq N$ we have
$N'\cap f(M)=0\implies N'=0$.

\begin{lemma} \label{lemma:Jessential}
	A \jmap{} $\varphi:M\to N$ is essential if and only if
	$\restr{\varphi}{M[J]}:M[J]\to N[J]$ is.
\end{lemma}
\begin{proof}
	Notice that the statement is trivially true in case $J=0$, so we
	may assume that $J\neq 0$.
	If $\varphi$ is essential then clearly so is $\restr \varphi{M[J]}$,
	because any submodule $N'$ of $N[J]$ such that $N'\cap \varphi(M[J])=0$ 
	is in particular a submodule of $N$ such that $N'\cap \varphi(M)=0$.

	Assume than that $\restr \varphi{M[J]}:M[J]\to N[J]$ is essential.
	Let $N'\subseteq N$ be a non-trivial submodule and let $n\in N'$ be a
	non-zero element. If $n\in N[J]$ then $N'\cap N[J]$ is non-trivial, and
	since $\restr \varphi{M[J]}$ is essential then $N'\cap \varphi(M)[J]$ is
	non-trivial as well. So we may assume that $n\not\in N[J]$.

	Since $\varphi:M\to N$ is a \jmap{}, there is $I\in J$ such that
	$In\subseteq\varphi(M)$. In particular, since $0\not \in J$ and $n$ is
	not $J$-torsion, there is $r\in R$ such that $0\neq rn\in \varphi(M)$.
	Since $N'$ is a submodule we have $rn\in N'\cap \varphi(M)$,
	so $\varphi:M\to N$ is an essential extension.
\end{proof}

\begin{lemma} \label{lemma:diffIsTorsion}
	Let $\varphi:M\to N$ be a $J$-map and let $f,g:N\to P$ be
	$R$-module homomorphisms such that $f\circ\varphi=g\circ\varphi$.
	Then for every $n\in N$ we have $f(n)-g(n)\in P[J]$.
\end{lemma}
\begin{proof}
	The statement is clearly true for $J=0$, so we may assume that $J\neq 0$.
	Since $\divm J {\varphi(M)}N=N$ there is $I\in J$ such that
	$In\subseteq \varphi(M)$. In particular there is a non-zero
	$r\in I$ such that $rn\in \varphi(M)$, say $rn=\varphi(m)$ for some
	$m\in M$. This implies that
	\begin{align*}
		r(f(n)-g(n))=f(\varphi(m))-g(\varphi(m)) = 0
	\end{align*}
	thus $f(n)-g(n)\in P[J]$.
\end{proof}

\subsection{\jinj{} modules and \jhulls{}}

Fix for this section a ring $R$ and a complete ideal filter $J$ of $R$.
We introduce the notion of \emph{\jinj{} module}, which generalizes the
classical notion of injectivity.

\begin{definition}
	A left $R$-module $Q$ is called \emph{\jinj{}}
	if for every \jextension{}
	$i:M\into N$ and every $R$-module homomorphism $f:M\to Q$ there
	exists a homomorphism $g:N\to Q$ such that $g\circ i=f$.
\end{definition}

\begin{remark}
	Notice that in case $J=0$ the definition of $J$-injective $R$-module
	coincides with that of injective module.
	Moreover, if $J'$ is a complete ideal filter of $R$ such that
	$J'\subseteq J$, then a $J$-injective module is also $J'$-injective.
\end{remark}


\begin{example}
	A $\Z$-module is $p^\infty$-injective if and only if it is $p$-divisible
	as an abelian group. The proof of this fact is completely analogous to
	that of the well-known result that a $\Z$-module is injective if and only
	if it is divisible.
\end{example}

The following proposition is an analogue of the well-known Baer's criterion
in the classical case of injective modules.

\begin{proposition} \label{prop:baerCriterion}
	A left $R$-module $Q$ is \jinj{} if and only if for every two-sided
	ideal $I\in J$ and every $R$-module homomorphism $f:I\to Q$ there is
	an $R$-module homomorphism $g:R\to Q$ that extends $f$.
\end{proposition}

\begin{proof}
	The ``only if'' part is trivial, because any two-sided ideal of $R$
	is also a left $R$-module and $I\into R$ is a \jextension{} if
	$I\in J$.  For the other implication, let $i:M\into N$ be a
	\jextension{} and let $f:M\to Q$ be any $R$-module homomorphism.
	By Zorn's Lemma there is a submodule $N'$ of $N$ and an extension
	$g':N'\to Q$ of $f$ to $N'$ that is maximal in the sense that 
	it cannot be extended to any larger submodule of $N$. If $N'=N$ we are 
	done, so assume that $N'\neq N$ and let $x\in N\setminus N'$.

	Let $I$ be the two-sided ideal of $R$ generated by $\set{r\in R\mid rx\in 
	N'}$. Since $i(M)\subseteq N'$ and $\divm J{i(M)}N=N$ there is $I'\in J$
	such that $I'x\subseteq N'$, which implies $I'\subseteq I$, so also
	$I\in J$.  By assumption the map $I\to Q$ that sends 
	$y\in I$ to $g'(yx)$ extends to a map $h:R\to Q$. Since $\ker(R\to Rx)$ is 
	contained in $\ker(h)$, the map $h$ gives rise to a map $h':Rx\to Q$ by 
	sending $rx\in Rx$ to $h(r)$. By definition the restrictions of $g'$ and 
	$h'$ to $N'\cap Rx$ coincide, so we can define a map $g'':N'+Rx\to Q$ that 
	extends both. This contradicts the maximality of $g'$, so we conclude that 
	$N'=N$.
\end{proof}

\begin{remark}
	Let $R$ be an integral domain and let $J$ be the ideal filter $0$ on $R$.
	Then the set of ideals $J'=J\setminus\{0\}$ is an ideal filter.
	Using Proposition \ref{prop:baerCriterion} one can easily show that an
	$R$-module $Q$ is $J$-injective if and only if it is $J'$-injective.
	Indeed, one implication holds, as remarked above, because $J\subseteq J'$,
	and for the other it is enough to notice that the unique map $0\to Q$ can
	always be extended to the zero map on $R$.

	One advantage of using $J'$ instead of $J$ is that the $J'$-torsion
	submodule may be different from the whole module.
\end{remark}

\begin{example} \label{example:localizationjinj}
	Let $M$ be an abelian group, let $p$ be a prime and let $J=p^\infty$ be
	the ideal filter of $\Z$ introduced in Example \ref{example:pInfty}.
	Then the localization $M[p^{-1}]$ is a \jinj{} $\Z$-module.
	Indeed if $i:N\into P$ is a \jextension{} and $f:N\to M[p^{-1}]$ is any
	homomorphism then for every $x\in P$ there is $k\in\NN$ such that
	$p^kx\in i(N)$, and one can define $g(x):=\frac{f(p^kx)}{p^k}$.
	It is easy to check that $g$ is then a well-defined group homomorphism
	such that $g\circ i=f$.
\end{example}

\begin{proposition} \label{proposition:EssentialExtensionOfInjective}
	Let $M$ be a \jinj{} $R$-module.
	If $f:M\into N$ is an essential $J$-extension, then it is an isomorphism.
\end{proposition}
\begin{proof}
	By definition of $J$-injectivity
	there is a map $g:N\to M$ such that $g\circ f= \id_M$.
	Then $g$ is surjective and since $f$ is an essential extension
	then $g$ is also injective, so it is an isomorphism.
\end{proof}

Recall that an \emph{injective hull} of an $R$-module $M$ is an essential
extension $i:M\into N$ such that $N$ is injective as an $R$-module.
It is well-known that every $R$-module $M$ admits
an injective hull and that any two injective hulls $i:M\into \Omega$ and
$j:M\into \Gamma$ are isomorphic via a (not necessarily unique) isomorphism
that commutes with $i$ and $j$, see \cite{baer}, \cite{injektive} or
\cite{fleischer}.

\begin{lemma} \label{lemma:InjectiveHull}
	Let $R$ be a ring and let $M$ be a left $R$-module. If $i:M\into
	\Omega$ is an injective hull and $j:M\into N$ is an essential extension,
	there is an injective $R$-module homomorphism $\varphi:N\into\Omega$ such
	that $\varphi\circ j=i$. Moreover, $\varphi:N\into\Omega$ is an injective
	hull.
\end{lemma}
\begin{proof}
	Since $\Omega$ is injective there exists an $R$-module homomorphism 
	$\varphi: N\to \Omega$ such that $\varphi\circ j=i$. Since $i$ is injective 
	and $j$ is an essential extension, then also $\varphi$ is injective.

	The last part follows from the fact that $\Omega$ is injective and
	$\varphi: N\into \Omega$ is an essential extension, since
	$i:M\into \Omega$ is.
\end{proof}

We conclude this section by proving that every $R$-module admits a
\emph{\jhull{}}, which is the generalization of an injective hull:

\begin{definition}
	Let $M$ be a left $R$-module. A \jextension{} $\iota:M\into \Omega$ is 
	called a \emph{\jhull{}} of $M$ if it is an essential extension and 
	$\Omega$ is \jinj{}.
\end{definition}

\begin{remark}
	If $J=0$ the definition of \jhull{} coincides with that of injective hull.
\end{remark}

\begin{remark} \label{remark:jhullsum}
	If $f_i:M_i\into N_i$, for $i=1,\dots, k$, are \jhulls{}, then the finite
	sum
	\begin{align*}
		\oplus_if_i:\bigoplus_{i=1}^kM_i\into \bigoplus_{i=1}^k N_i
	\end{align*}
	is a \jhull{}. Indeed $\bigoplus_iN_i$ is \jinj{} because it is a finite
	direct sum of $J$-injective modules, and it is easy to see that it is
	also an essential \jextension{} of $\bigoplus_iM_i$.
\end{remark}

\begin{lemma} \label{lemma:InjectiveSub}
	Let $Q$ be a \jinj{} $R$-module and let $P\subseteq Q$ be any
	submodule. Then $\divm JPQ$ is \jinj{}.
\end{lemma}
\begin{proof}
	Let $i:M\into N$ be a \jextension{} and let $f:M\to \divm JPQ$ be
	any $R$-module homomorphism. Denote by $j:\divm JPQ\into Q$ the
	inclusion. Since $Q$ is \jinj{}, there is a map $g:N\to Q$ such that
	$g\circ i = j\circ f$. For every $x\in N$ there is some $I\in J$
	such that $Ix\subseteq i(M)$ and thus $Ig(x)=g(Ix)\subseteq 
	g(i(M))=j(f(M))$, which means that the image of $g$ is contained in
	$\divm JPQ$. This shows that $\divm JPQ$ is \jinj{}.
\end{proof}

\begin{theorem} \label{theorem:Jhull}
	Every left $R$-module $M$ admits a \jhull{}. Moreover,
	the following holds for any \jhull{} $\iota:M\into \Omega$ of $M$:
	\begin{enumerate}
		\item For every \jextension{} $i:M\into N$ there is a
		      \jhull{} $j:N\into\Omega$ with $j\circ i=\iota$.
		\item For every \jhull{} $\iota':M\into\Omega'$ there is an
		      isomorphism $\varphi:\Omega \isomto \Omega'$ with
		      $\varphi\circ \iota=\iota'$.
	\end{enumerate}
\end{theorem}

\begin{proof}
	Let $\iota:M\into\Gamma$ be an injective hull of $M$ and let 
	$\Omega:=\divm{J}{\iota(M)}{\Gamma}$.
	Since $\iota:M\into\Gamma$ is an essential extension
	then also $\iota:M\into\Omega$ is, and by Lemma 
	\ref{lemma:propertiesOfDivision}(2) we have $\divm{J}{\iota(M)}{\Omega}= 
	\Omega$, so $\iota:M\into\Omega$ is a \jextension{} of $M$.
	By Lemma \ref{lemma:InjectiveSub} the $R$-module $\Omega$ is \jinj{},
	so it is a \jhull{} of $M$.

	For (1), since $\Omega$ is \jinj{} there is a map 
	$j:N\to \Omega$ such that $j\circ i=\iota$. Moreover since $\iota:M\into 
	\Omega$ is an essential extension also $j:N\into \Omega$ is, so it is a 
	\jhull{}.

	For (2), let $\iota:M\into\Omega$ and $\iota':M\into\Omega'$ be two 
	\jhulls{}. Since $\Omega'$ is \jinj{} there is an $R$-module homomorphism 
	$f:\Omega\to \Omega'$ such that $f\circ\iota=\iota'$, so since $\iota$ is
	an essential extension $f$ is injective. But
	then, since $\id_\Omega:\Omega\into\Omega$ is a \jhull{} by (1), there is 
	an $R$-module homomorphism $g:\Omega'\to \Omega$ such that $g\circ 
	f=\id_\Omega$, so in particular $g$ is surjective. But we also have
	$g\circ \iota'=\iota$, and since $\iota'$ is an essential extension then
	$g$ must be injective too, hence it is an isomorphism.
\end{proof}

\begin{example}
	Let $M$ be a finitely generated abelian group,
	let $p$ be a prime number and let $J=p^\infty$ be
	the ideal filter of $\Z$ introduced in Example \ref{example:pInfty}.
	Write $M$ as
	\begin{align*}
		M=\Z^r \oplus \bigoplus_{i=1}^k \Z/p^{e_i}\Z\oplus M[n] 
	\end{align*}
	where $n$ is a positive integer coprime to $p$ and the $e_i$'s are suitable
	exponents. Let
	\begin{align*}
		\Gamma=(\Z[p^{-1}])^r\oplus(\Z[p^{-1}]/\Z)^k\oplus M[n]
	\end{align*}
	and
	\begin{align*}
		\begin{array}{lccc}
			\iota: & M & \to & \Gamma \\
				   & (z, (s_i\bmod{p^{e_i}})_{i}, t) & \mapsto &
				     \left(\frac z1,
					 \left(\frac{s}{p^{e_i}}\bmod{\Z}\right)_{i} ,
				     t\right)
		\end{array}
	\end{align*}
	Then $\iota:M\to \Gamma$ is a \jhull{}. To see this it is enough to
	show that $f:\Z^r\into (\Z[p^{-1}])^r$ and
	$g_i:\Z/p^{e_i}\Z\into \Z[p^{-1}]/\Z$ for every $i=1,\dots, k$ are
	\jhulls{}, and that $M[n]$ is \jinj{}, being trivially an essential
	extension of itself. The assertions about $f$ and $M[n]$ follow from
	Example \ref{example:localizationjinj}, noticing that multiplication by
	$p$ is an automorphism of $M[n]$ and that $\Z^r\into (\Z[p^{-1}])^r$ is an
	essential \jextension{}.

	So we are left to show that for every positive integer $e$ the map
	$g:\Z/p^e\Z\into \Z[p^{-1}]/\Z$ defined by
	$(s\bmod{p^e})\mapsto (\frac{s}{p^e}\bmod{\Z})$ is a \jhull{}.
	It is a \jextension{}, because the Prüfer group $\Z[p^{-1}]/\Z$ itself
	is $J$-torsion, and it is also essential because every subgroup of 
	$\Z[p^{-1}]/\Z$ is of the form $\frac{1}{p^d}\Z$, so it intersects the
	image of $g$ in $\frac{1}{p^{\min(e,d)}}\Z$.

	Finally, $\Z[p^{-1}]/\Z$ is divisible as an abelian group, so in
	particular it is \jinj{}, since in this case it is equivalent
	to being $p$-divisible.
\end{example}

\subsection{Duality}
\label{section:duality}

Fix again a ring $R$ and a complete ideal filter $J$ of $R$.
Fix as well a left $R$-module $M$ and a $J$-injective and $J$-torsion
left $R$-module $T$ and let $E=\End_R(T)$.

In this section we prove an elementary duality result that will be key
to the proof of our main Kummer-theoretic results
(Theorem \ref{thm:sesCohomology}).

\begin{definition}
	If $V$ is a subset of $\Hom_R(M,T)$ we
	denote by $\ker(V)$ the submodule of $M$ given by
	\begin{align*}
		\ker(V) := \bigcap_{f\in V} \ker(f)
	\end{align*}
	and we call it the \emph{joint kernel} of $V$.
\end{definition}

If $M'$ is a submodule of $M$ we will identify
$\Hom_R(M/M',T)$
with
the submodule $\set{f\in\Hom_R(M,T)\mid \ker(f)\supseteq M'}$
of $\Hom_R(M,T)$.

\begin{proposition} \label{prop:duality1}
	If $V$ is a finitely generated $E$-submodule of $\Hom_R(M,T)$ we have
	$V=\Hom_R(M/\ker(V), T)$.
\end{proposition}
\begin{proof}
	Notice that the inclusion $V\subseteq \Hom_R(M/\ker(V),T)$ is obvious.
	For the other inclusion we want to show that every homomorphism $g:M\to T$
	with $\ker(g)\supseteq \ker(V)$ belongs to $V$. Let then $g$ be such a map
	and let $\overline g:M/\ker(V)\to T$ be its factorization through the
	quotient $M/\ker(V)$.
	Let $\set{f_1,\dots,f_n}$ be a set of generators for $V$ as an
	$E$-module and let
	\begin{align*}
		\varepsilon: M & \to T^n \\
		x & \mapsto (f_1(x), \dots, f_n(x))
	\end{align*}
	We have $\ker(\varepsilon)=\ker(V)$, so that $\varepsilon$ factors
	as an injective map $\overline \varepsilon:M/\ker(V)\to T^n$. Since
	$T$ is $J$-torsion, so is $T^n$, hence $\overline\varepsilon$ is
	a $J$-extension. Since $T$ is $J$-injective there is an $R$-linear map
	$\lambda:T^n\to T$ such that $\lambda\circ\overline\varepsilon=\overline g$,
	or equivalently $\lambda\circ \varepsilon=g$.
	\begin{equation*}
		\begin{tikzcd}
			& & T \\
			M \ar[r, two heads] \ar[drr, bend right, "\varepsilon", swap]
				\ar[rru, bend left, "g"] &
			M/\ker(V) \ar[ru, "\overline g"]
				\ar[dr, swap, "\overline\varepsilon", hook] \\
			& & T^n \ar[uu, dashed, "\lambda", swap]
		\end{tikzcd}
	\end{equation*}
	Since $\Hom_R(T^n,T)\cong \bigoplus_{i=1}^n\End_R(T)$, there are elements
	$e_1,\dots,e_n\in \End_R(T)$ such that $\lambda(t_1,\dots,t_n)=
	e_1(t_1)+\dots +e_n(t_n)$ for every $(t_1,\dots,t_n)\in T^n$.
	Then for $x\in M$ we get
	\begin{align*}
		\lambda(\varepsilon(x)) & = \lambda(f_1(x),\dots,f_n(x))\\
		& = e_1(f_1(x)) + \cdots + e_n(f_n(x))
	\end{align*}
	which means that $g=e_1\circ f_1+\cdots+e_n\circ f_n\in V$ because
	$V$ is an $E$-module.
\end{proof}

\begin{remark}
	Proposition \ref{prop:duality1} is a generalization of the following
	fact from linear algebra: let $V$ be a finite-dimensional vector space
	over a field $K$ and let $f_1,\dots,f_n:V\to K$ be linear functions.
	If $f:V\to K$ is a linear function such that $\ker(f)\supseteq
	\bigcap_{i=1}^n \ker(f_i)$, then $f$ is a linear combination of
	$f_1,\dots,f_n$.
\end{remark}

\begin{definition}
	Let $N$ and $Q$ be left $R$-modules. We say that $Q$ is a
	\emph{cogenerator} for $N$ if $\ker(\Hom_R(N,Q))=0$.
\end{definition}

\begin{theorem} \label{thm:duality}
	Let $R$ be a ring and let $J$ be a complete ideal filter on $R$.
	Let $T$ be a $J$-injective and $J$-torsion left
	$R$-module and let $M$ be any left $R$-module. Assume that $T$ is a
	cogenerator for every quotient of $M$ and that $\Hom_R(M,T)$ is
	Noetherian as an $\End_R(T)$-module. The maps
	\begin{align*}
		\begin{array}{ccc}
			\left\{R\text{-submodules of }M\right\}
			& \to & 
			\left\{\End_R(T)\text{-submodules of }\Hom_R(M,T)\right\} \\
			M' & \mapsto & \Hom_R(M/M', T)\\
			\ker(V) & \mapsfrom & V
		\end{array}
	\end{align*}
	define an inclusion-reversing bijection between the set of
	$R$-submodules of $M$ and that of $\End_R(T)$-submodules of
	$\Hom_R(M,T)$.
\end{theorem}
\begin{proof}
	Notice first of all that the maps are well-defined and they are
	both inclusion-reversing.
	Since $\Hom_R(M,T)$ is Noetherian as an $\End_R(T)$-module, every
	submodule is finitely generated, so we may apply Proposition
	\ref{prop:duality1}. Since $T$ is a cogenerator for every quotient
	of $M$ we can conclude that the two given maps are inverse of
	each other.
\end{proof}

\begin{example}
	Let $R=\Z$, let $J=\infty$ and let $T=(\Q/\Z)^s$ for some positive integer
	$s$. Let $M$ be a finitely generated abelian group.
	Notice that $T$ is $J$-torsion and, since it is injective,
	it is in particular $J$-injective. Since $\Q/\Z$ is a cogenerator for every
	abelian group, then so is $T$.
	We have $\End_R(T)=\Mat_{s\times s}(\hat \Z)$ and since
	$M$ is finitely generated $\Hom_R(M,T)$ is Noetherian over
	$\Mat_{s\times s}(\hat \Z)$.
	We are then in the setting of Theorem \ref{thm:duality}. 
\end{example}

\section{The category of \jtextensions{}}
\label{section:jtext}

Fix for this section a ring $R$, a complete ideal filter $J$
of $R$ and a $J$-torsion and \jinj{} left $R$-module $T$.  

In this section we introduce \jtextensions{}, which are essentially
\jextensions{} whose $J$-torsion is contained in an $R$-module $T$ as above
(see Definition \ref{definition:jtext}).
These extensions of $R$-modules share many interesting properties with
field extensions, and in fact at the end of this section we will be able to
prove the existence of a ``maximal'' \jtextension{}, analogous to an
algebraic closure in field theory.

\subsection{\pointed{} $R$-modules}

In order to define \jtextensions{} we first introduce the more fundamental
concept of \pointed{} $R$-module.

\begin{definition}
	A \emph{\pointed{} $R$-module} is a pair $(M,s)$, where $M$ is a left
	$R$-module and $s:M[J]\into T$ is an injective homomorphism.

	If $(L,r)$ and $(M,s)$ are two \pointed{} $R$-modules, we call an
	$R$-module homomorphism $\varphi:L\to M$ a \emph{homomorphism} or
	\emph{map of \pointed{} $R$-modules} if $s\circ\restr{\varphi}{L[J]} = r$.
\end{definition}

In the following we will sometimes omit the map $s$ from the notation and
simply refer to \emph{the \pointed{} $R$-module $M$}.

\begin{remark} \label{remark:TPRMmapInjectiveOnTorsion}
	A map $\varphi:(L,r)\to (M,s)$ of \pointed{} $R$-modules is injective
	on $L[J]$. Indeed $s\circ\restr\varphi{L[J]}=r$ is injective, so
	$\restr\varphi{L[J]}$ must be injective as well.
\end{remark}

\begin{definition}
	If $(M,s)$ is a \pointed{} $R$-module we denote the \pointed{} $R$-module
	$(M[J],s)$ by $\jt(M,s)$, or simply by $\jt(M)$. We will denote the natural
	inclusion $\jt(M)\into M$ by $\inct{M}$.
\end{definition}

\begin{example}
	Let $R=\Z$ and let $J$ be the complete ideal filter $\infty$ on $\Z$.
	Let $T=(\Q/\Z)^2$, which is  $\infty$-injective and $\infty$-torsion.
	The abelian group $M=\Z\oplus \Z/6\Z\oplus \Z/2\Z$ together with the
	map $s:\Z/6\Z\oplus \Z/2\Z$ that sends $(1,0)$ to $\left(\frac16,0\right)$
	and $(0,1)$ to $\left(0,\frac12\right)$ is a $T$-pointed $R$-module.
\end{example}

As is the case with field extensions, pushouts do not always exist in
our newly-defined category. However the pushout of two maps 
of $T$-pointed $R$-modules exists if at least
one of the two is injective and ``as little a \jmap{} as possible''.

\begin{definition}
	We say that a map $f:L\to M$ of \pointed{} $R$-modules is \emph{pure}
	if $\divm{J}{f(L)}{M}=f(L)+M[J]$. 
\end{definition}

\begin{proposition}
	\label{prop:pushout}
	Let $(L,r)$, $(M,s)$ and $(N,t)$ be \pointed{} $R$-modules and let
	$f:L\to M$ and $g:L\to N$ be maps of \pointed{} $R$-modules.
	Assume that $f$ is injective and pure. Then the pushout
	\begin{tikzcd}
		M \ar[r,"i"] & P & \ar[l,"j",swap] N
	\end{tikzcd}
	of $f$ along $g$ exists in the category of \pointed{} $R$-modules.

	Moreover the pushout map $j:N\to P$ is injective, and if $g$ is
	injective the pushout map $i:M\to P$ is injective.
\end{proposition}
\begin{proof}
	We have to show that there is a \pointed{} $R$-module $(P,u)$ with maps
	$i:M\to P$ and $j:N\to P$ such that the diagram
	\begin{center}
		\begin{tikzcd}
			L \ar[d, swap, "g"] \ar[r, "f"] & M \ar[d, "i"] \\
			N \ar[r, swap, "j"] & P
		\end{tikzcd}
	\end{center}
	commutes and such that for every \pointed{} $R$-module $(Q,v)$ with maps
	$k:M\to Q$ and $l:N\to Q$ with $k\circ f=l\circ g$ there is a unique map
	$\varphi:L\to Q$ such that the diagram
	\begin{center}
		\begin{tikzcd}
			L \ar[d, swap, "g"] \ar[r, "f"] & M \ar[d, "i"] 
				\ar[ddr, bend left, "k"] \\
			N \ar[r, swap, "j"] \ar[drr, bend right, swap, "l"] & P 
				\ar[dr, "\varphi"]\\
			& & Q
		\end{tikzcd}
	\end{center}
	commutes.

	Let $P'$ be the pushout of $f$ along $g$ as maps of $R$-modules, and let
	$i':M\to P'$ and $j':N\to P'$ be the pushout maps.  Write $P'$ as
	$(M\oplus N)/S$ where $S=\set{(f(\lambda),-g(\lambda))\mid \lambda\in L}$.
	Let $\pi:P'\to P$ be the quotient by the submodule
	\begin{align*}
		K:=\left\langle\set{[(m,-n)]\mid \text{for all }m\in M[J],\,n\in N[J]
			\text{ such that } s(m)=t(n)}\right\rangle
	\end{align*}
	and let $i=\pi\circ i'$ and $j=\pi\circ j'$. Notice that
	$i\circ f =j\circ g$.

	We claim that $P'[J]$ is generated by $i'(M[J])$ and $j'(N[J])$.
	The claim is obviously true if $J=0$, so we may assume that $J\neq 0$.
	To prove the claim, notice that by Lemma \ref{lemma:propertiesOfDivision}(3)
	we have $P'[J]=\divm{J}{S}{M\oplus N}/S$, so any element of $P'[J]$ is
	represented by a pair $(m,n)$ such that $I(m,n)\subseteq S$ for some
	$I\in J$. Then since $f$ is a pure map
	we have $m=f(\lambda)+t_m$ for some $\lambda\in L$ and some $t_m\in M[J]$.

	Let $I'\in J$ be such that $I't_m = 0$.
	Then $I\cap I'\in J$ and for any nonzero $h\in I\cap I'$ we have
	$(f(h\lambda),hn)=h(m-t_m,n)=h(m,n)\in S$, which means that
	$hn=-g(h\lambda + z)$ for some $z\in\ker(f)$. Since $f$ is injective
	we have that $n=-g(\lambda)+t_n$ for some $t_n\in N[J]$. It
	follows that the class of $(m,n)$ in $P'[J]$ is the same as that
	of $(t_m,t_n)$, which proves our claim.

	Since $K\subseteq P'[J]$, it follows easily from our claim that
	$P[J]=P'[J]/K$ and thus that the map
	\begin{align*}
		u:P[J]&\to T \\
		[(m,n)]&\mapsto s(m)+t(n)
	\end{align*}
	is well-defined and injective.
	This shows that $(P,u)$ is a \pointed{} $R$-module and that $i:M\to P$ and
	$j:N\to P$ are maps of \pointed{} $R$-modules.

	Let now $(Q,v)$, $k$ and $l$ be as above. By the universal property of the
	pushout there is a unique $R$-module homomorphism $\varphi':P'\to Q$ such
	that $\varphi'\circ i'=k$ and $\varphi'\circ j'=l$. Since $k$ is a map
	of \pointed{} $R$-modules, this implies that $v\circ\varphi'\circ i'=s$
	and $v\circ\varphi'\circ j'=t$, so that $\varphi'$ factors through $P$
	as a \pointed{} $R$-module homomorphism $\varphi:P\to Q$.

	For the last assertion we first notice that if $g$ is injective, then so is
	the $R$-module pushout map $i'$. Then we claim that $i'(M)\cap K=0$. Indeed
	if $[(m_0,0)] = [(m,-n)]$ in $P'$ for some $m_0\in m$, $m\in M[J]$ and
	$n\in N[J]$ such that $s(m)=t(n)$, then there is some $\lambda\in L$ such
	that $m-m_0=f(\lambda)$ and $n=g(\lambda)$. Since $g$ is injective,
	$\lambda$ is $J$-torsion, and we have $r(\lambda)=s(m)-s(m_0)=t(n)$. But
	since $s(m)=t(n)$ we must have $m_0=0$, and we conclude that
	$i'(M)\cap K=0$. It follows that $i=\pi\circ i'$ is injective. The fact
	that the injectivity of $f$ implies that of $j$ is completely analogous.
\end{proof}

\begin{remark} \label{remark:noPushout}
	Let $R=\Z$, $J=2^\infty$, $T=\Z\left[\frac12\right]/\Z$, $L=\Z$ and
	$M=N=\frac{1}{2}\Z$. The $R$-modules $L$, $M$ and $N$ are $T$-pointed
	via the zero map, since their $J$-torsion is trivial.
	Let $f:L\into M$ and $g:L\into N$ be the natural inclusion and notice
	that they are maps of $T$-pointed $R$-modules that are not pure.
	We claim that the pushout of $f$ along $g$ does not exist in the
	category of $T$-pointed $R$-modules.

	Suppose instead that
	$(P,u)$ is a pushout of
	$f$ along $g$ and consider the $T$-pointed $R$-module
	$\left(\frac12\Z\oplus \Z/2\Z, z\right)$, where $z:\Z/2\Z\to T$ is the
	only possible injective map. Consider the diagram
	\begin{center}
		\begin{tikzcd}
			L \ar[d, swap, "g"] \ar[r, "f"] & M \ar[d, "i"] 
				\ar[ddr, bend left, "k"] \\
			N \ar[r, swap, "j"] \ar[drr, bend right, swap, "l"] & P 
				\ar[dr, "\varphi"]\\
			& & \frac12\Z \oplus \frac{\Z}{2\Z}
		\end{tikzcd}
	\end{center}
	where the maps $k$ and $l$ are defined as
	\begin{align*}
		\begin{array}{ccccccccccc}
		k:&\frac12\Z & \to & \frac12\Z\oplus \frac{\Z}{2\Z} & & & &
		l:&\frac12\Z & \to & \frac12\Z\oplus \frac{\Z}{2\Z} \\
		 & & & & \qquad & \text{and} & \qquad\\
		&\frac12 & \mapsto & \left(\frac12,0\right) & & & &
		&\frac12 & \mapsto & \left(\frac12,1\right) \\
		\end{array}
	\end{align*}
	Notice that $k$ and $l$ are maps of $T$-pointed $R$-modules such that
	$k\circ f=l\circ g$. Then by assumption there exists a unique map
	of $T$-pointed $R$-modules $\varphi:P\to \frac12\Z\oplus \Z/2\Z$ that
	makes the diagram commute.
	In particular we have $\varphi(j(\frac12))\neq\varphi(i(\frac12))$,
	which implies that $j(\frac12)\neq i(\frac12)$. But since
	$2j(\frac12)=j(g(1))=i(f(1))=i(\frac12)$ we have that
	$t:=j(\frac12)-i(\frac12)$ is a $2$-torsion element of $P$, and we
	must have $u(t)=\frac12$.

	Consider now the map $k':M\to \frac12\Z\oplus \Z/2\Z$ mapping $\frac12$
	to $\left(\frac12,0\right)$, just as $l$ does. This is again a map
	of $T$-pointed $R$-modules such that $k'\circ f=l\circ g$, so there
	must be a map of $T$-pointed $R$-modules $\varphi':P\to\frac12\Z\oplus
	\Z/2\Z$ that makes this new diagram commute. Such a map $\varphi'$
	must map $t$ to $0$, because $\varphi'(j(\frac12))=
	\left(\frac12,0\right)=\varphi'(i(\frac12))$. But then the diagram of
	structural maps into $T$
	\begin{center}
		\begin{tikzcd}
			P[J] \ar[dr,"u"]
				\ar[dd,swap,"\restr{\varphi'}{P[J]}"]\\
			& T\\
			\frac{\Z}{2\Z} \ar[ur,"z",swap]
		\end{tikzcd}
	\end{center}
	would not commute, which is a contradiction. This proves our claim.
\end{remark}


The class of $T$-pointed $R$-modules whose torsion submodule is
isomorphic to $T$ will be particularly important for us.

\begin{definition}
	Let $(M,s)$ be a \pointed{} $R$-module. We say that $(M,s)$ is
	\emph{saturated} if $\inct M:M[J]\into T$ is surjective
	(and hence an isomorphism).
\end{definition}

\begin{remark}
	The map $\inct M$ is a pure and injective map.
\end{remark}

Every \pointed{} $R$-module can be embedded in a saturated module, and the smallest
saturated module containing a given one can be constructed as a pushout.

\begin{definition} \label{definition:saturationTpoint}
	If $(M,s)$ is a \pointed{} $R$-module we call \emph{saturation} of
	$(M,s)$, denoted by $\tp(M,s)$ or simply by $\tp(M)$, the \pointed{}
	$R$-module $(P,u)$ which is the pushout (in the category of \pointed{}
	$R$-modules) of the diagram
	\begin{center}
		\begin{tikzcd}
			M[J] \ar[d,"s"] \ar[r,hook, "\inct M"] & M \ar[d, "\incs M"] \\
			T \ar[r] & P
		\end{tikzcd}
	\end{center}
	We will also denote by $\tp(s)$ the map $u$ and by $\incs M$ the pushout
	map $M\to P$.
\end{definition}

\begin{remark}
	Notice that the pushout map $T\to P$ of Definition
	\ref{definition:saturationTpoint} is an isomorphism onto $P[J]$.
	Indeed by definition of $T$-pointed $R$-module the following
	diagram commutes:
	\begin{center}
		\begin{tikzcd}
			T=T[J]  \ar[dr, "\id_T"] \ar[dd] \\
			& T \\
			P[J] \ar[ur,swap,"\tp(s)"]
		\end{tikzcd}
	\end{center}
	where the vertical map on the left is the pushout map.
	It follows that $\tp(s)$, which is injective by definition,
	is also surjective, hence an isomorphism, and the pushout map
	is its inverse. In other words, the saturation of a $T$-pointed
	$R$-module is saturated.
\end{remark}

\subsection{\jtextensions{}}

We can finally introduce the main object of study of this section.

\begin{definition} \label{definition:jtext}
	Let $(M,s)$ be a \pointed{} $R$-module. A \emph{\jtextension{}} of $(M,s)$ 
	is a triple $(N,i,t)$ such that $(N,t)$ is a \pointed{} $R$-module and
	$i:M\into N$ is a map of \pointed{} $R$-modules and a \jextension{}.

	If $(N,i,t)$ and $(P,j,u)$ are two \jtextensions{} of $(M,s)$ we call a 
	homomorphism of \pointed{} $R$-modules $\varphi:N\to P$ a
	\emph{homomorphism} or \emph{map of \jtextensions{}} if $\varphi\circ i=j$.

	We denote by $\jtcat(M,s)$ the category of \jtextensions{} of $(M,s)$.
\end{definition}

In the following we will sometimes omit the maps $i$ and $t$ from the notation
and simply refer to \emph{the \jtextension{} $N$ of M}.

\begin{remark}
	Let $(N,i,t)$ and $(P,j,u)$ be \jtextensions{} of the \pointed{} $R$-module
	$(M,s)$ and let $\varphi:N\to P$ be a map of \jtextensions{}.
	Then $(P,\varphi,u)$ is a \jtextension{} of $(N,t)$. In fact we have
	\begin{align*}
		\divm{J}{\varphi(N)}{P} \supseteq \divm{J}{j(M)}{P}=P\,.
	\end{align*}
\end{remark}

\begin{example}
	Let $R=\Z$, let $J$ be the complete ideal filter $2^\infty$ of $\Z$
	and let $T$ be the $2^\infty$-torsion and $2^\infty$-injective
	$\Z$-module $\left(\Z\left[\frac12\right]/\Z\right)^2$.
	If $M=\Z\oplus \Z/2\Z \oplus \Z/2\Z$ then the map
	$s:\Z/2\Z\oplus \Z/2\Z\to T$ that sends $(1,0)$ to
	$\left(\frac12,0\right)$ and $(0,1)$ to $\left(0,\frac12\right)$ turns
	$(M,s)$ into a $T$-pointed $R$-module.

	Let $N=\frac12\Z\oplus \Z/4\Z\oplus \Z/2\Z$. The maps
	\begin{align*}
	\begin{array}{ccc}
		\begin{array}{ccc}
			t_1:\Z/4\Z \oplus \Z/2\Z & \to & T \\
			(1,0) & \mapsto & \left(\frac14,0\right)\\ 
			(0,1) & \mapsto & \left(0,\frac12\right) 
		\end{array}
		& \qquad\text{and}\qquad &
		\begin{array}{ccc}
			t_2:\Z/4\Z \oplus \Z/2\Z & \to & T \\
			(1,0) & \mapsto & \left(0,\frac14\right)\\
			(0,1) & \mapsto & \left(\frac12,0\right)
		\end{array}
	\end{array}
	\end{align*}
	define two different $T$-pointed $R$-module structures $(N,t_1)$ and
	$(N,t_2)$ on $N$. The component-wise inclusion $f:M\into N$ is a
	$2^\infty$ extension. Since it is compatible with all the maps to $T$,
	both $(N,f,t_1)$ and $(N,f,t_2)$ are $(2^\infty,T)$-extensions of
	$M$. They are not isomorphic as $(2^\infty,T)$-extensions, because
	they are not isomorphic as $T$-pointed $R$-modules.
\end{example}

We can immediately see some similarities between \jtextensions{} and field
extensions: every map is injective, and every surjective map is an isomorphism.

\begin{lemma} \label{lemma:jthomIsInjective}
	Every map of \jtextensions{} is injective.
\end{lemma}
\begin{proof}
	Let $(N,i,t)$ and $(P,j,u)$ be \jtextensions{} of the \pointed{} $R$-module
	$(M,s)$ and let $\varphi:N\to P$ be a map of \jtextensions{}.
	Let $n\in \ker\varphi$. Since $i:M\into N$ is a \jextension{} there is
	$I\in J$ such that $In\subseteq i(M)$.
	But since $j:M\into P$ is injective and $\varphi(In)=0$, we must 
	have $In=0$, hence $n$ is $J$-torsion.
	But since $\varphi$ is a map of \pointed{} $R$-modules it is injective on
	$M[J]$ (see Remark \ref{remark:TPRMmapInjectiveOnTorsion}) so $n=0$.
\end{proof}

\begin{corollary} \label{corollary:surjIsIsom}
	Every surjective map of \jtextensions{} is an isomorphism.
\end{corollary}
\begin{proof}
	Let $(N,i,t)$ and $(P,j,u)$ be \jtextensions{} of the \pointed{} $R$-module
	$(M,s)$ and let $\varphi:N\to P$ be a map of \jtextensions{}.
	In view of Lemma \ref{lemma:jthomIsInjective} it is enough to show that if
	$\varphi$ is an isomorphism of $R$-modules, then its inverse
	$\varphi^{-1}:P\isomto N$ is also a map of \jtextensions{}. But the fact
	that $\varphi^{-1}\circ j=i$ follows directly from $\varphi\circ i=j$ while
	$t=u\circ\restr \varphi{P[J]}^{-1}=u$ follows from
	$u\circ \restr\varphi{N[J]}=t$.
\end{proof}

\begin{proposition}
	Let $(M,s)$ be a $T$-pointed $R$-module, let $(N,i,t)$ be a
	\jtextension{} of $(M,s)$ and let $(P,j,u)$ be a \jtextension{}
	of $(N,t)$. Then $(P,j\circ i, u)$ is a \jtextension{} of $(M,s)$.
\end{proposition}
\begin{proof}
	The map $j\circ i$ is clearly a $J$-injective map of $T$-pointed
	$R$-modules, and it is a $J$-map by Lemma \ref{lemma:transAndGap}.
\end{proof}

\subsection{Pullback and pushforward}

One can recover much information about the \jtextensions{} of a certain
\pointed{} $R$-module by studying the extensions of its torsion submodule
and of its saturation -- see for example our construction of the maximal
\jtextension{} in 
Section \ref{sec:maxJText}. In order to study the relation
between these categories, we introduce the more general
pullback and pushforward functors which, interestingly, form an adjoint pair.

\begin{definition}
	If $\varphi:L\to M$ is a map of \pointed{} $R$-modules and
	$(N,i,t)$ is a \jtextension{} of $M$, we let
	\begin{align*}
		\varphi^*N:=\divm{J}{i(\varphi(L))}{N}\,,\qquad
		\varphi^*i:=\restr{i}{\varphi(L)}\,,\qquad
		\varphi^*t:=\restr{t}{(\varphi^*N)[J]}
	\end{align*}
	and we call them the \emph{pullback along $\varphi$} of $N$, $i$ and
	$t$ respectively.
\end{definition}

\begin{lemma} \label{lemma:pullback1}
	Let $\varphi:L\to M$ be a map of \pointed{} $R$-modules and let
	$(N,i,t)$ be a \jtextension{} of $M$.
	Then $(\varphi^*N, \varphi^*i, \varphi^*t)$
	is a \jtextension{} of $\varphi(L)$.
\end{lemma}
\begin{proof}
	Clearly $(\varphi^*N,\varphi^*t)$ is a \pointed{} $R$-module and
	\begin{align*}
		\varphi^*t\circ\restr{\varphi^*i}{\varphi(L)[J]} =
		t\circ \restr{i}{\varphi(L)[J]} =
		\restr{s}{\varphi(L)}
	\end{align*}
	so $\varphi^*i:(\varphi(L),\restr{s}{\varphi(L)})\to
	(\varphi^*N,\varphi^*t)$ is an injective map of \pointed{} $R$-modules.

	Moreover $\divm{J}{\varphi^*i(\varphi(L))}{\varphi^*N}=\varphi^*N$ by
	definition and by Lemma \ref{lemma:propertiesOfDivision}(2), so that
	$(\varphi^*N,\varphi^*i,\varphi^*t)$ is a \jextension{}.
\end{proof}

\begin{definition}
	If $\varphi:L\to M$ is a map of \pointed{} $R$-modules, $N$ and $P$ are
	\jtextensions{} of $M$ and $f:N\to P$ is a map of \jtextensions{}, the map
	\begin{align*}
		\restr{f}{\varphi^*N}:\varphi^*N\to\varphi^*P
	\end{align*}
	is a map of \jtextensions{} of $\varphi(L)$, which
	we denote by $\varphi^*f$.
\end{definition}

\begin{proposition} \label{prop:pullbackFunctor}
	Let $\varphi:L\to M$ be a map of \pointed{} $R$-modules. The diagram
	\begin{center}
		\begin{tikzcd}
			(N,i,t) \arrow[r,mapsto] \arrow[d,"f"] &
				(\varphi^*N,\varphi^*i,\varphi^*t) \ar[d,"\varphi^*f"]\\
			(P,j,u) \arrow[r,mapsto] & (\varphi^*P, \varphi^*j,\varphi^*u)
		\end{tikzcd}
	\end{center}
	defines a functor from $\jtcat{(M,s)}$ to
	$\jtcat{(\varphi(L),\restr{s}{\varphi(L)})}$.
\end{proposition}
\begin{proof}
	In view of Lemma \ref{lemma:pullback1} we only need to check that
	$\varphi^*$ behaves well with the respect to
	the composition of maps of \jtextensions{}. If
	\begin{align*}
		N\overset{f}{\longrightarrow} P \overset{g}{\longrightarrow} Q
	\end{align*}
	are maps of \jtextensions{} of $(M,s)$, we have
	\begin{align*}
		\varphi^*g\circ\varphi^*f =
		\restr{g}{\varphi^*P}\circ \restr{f}{\varphi^*N}=
		\restr{(g\circ f)}{\varphi^*N}=\varphi^*(g\circ f)\,.
	\end{align*}
\end{proof}

\begin{definition}
	We call the functor of Proposition \ref{prop:pullbackFunctor} the
	\emph{pullback along $\varphi$}, and we denote it by $\varphi^*$.
\end{definition}

\begin{definition}
	If $\varphi:L\to M$ is an injective and pure map of \pointed{} $R$-modules
	and $(N,i,t)$ is a \jtextension{} of $L$ we denote by $\varphi_*i:M\to
	\varphi_*N$ the pushout of $i$ along $\varphi$.
\end{definition}

\begin{lemma} \label{lemma:pushforward1}
	Let $\varphi:L\to M$ be an injective and pure map of \pointed{}
	$R$-modules and let $(N,i,t)$ be a \jtextension{} of $L$.
	Then $(\varphi_*N, \varphi_*i, \varphi_*t)$ is a \jtextension{} of $(M,s)$.
\end{lemma}
\begin{proof}
	This follows from the fact that $\varphi_*i$ is injective and
	$\varphi_*N/(\varphi_*i)(M)\cong N/i(L)$
	is $J$-torsion, because $i:L\to N$ is a \jextension{}.
\end{proof}

\begin{lemma} \label{lemma:pushforward2}
	Let $\varphi:L\to M$ be an injective and pure map of \pointed{} $R$-modules,
	let $(N,i,t)$ and $(P,j,u)$ be \jtextensions{} of $L$ and let
	$f:N\to P$ be a map of \jtextensions{}. Then there is a unique map
	of \jtextensions{} of $M$
	\begin{align*}
		\varphi_*f:\varphi_*N\to\varphi_*P
	\end{align*}
	such that the diagram
	\begin{center}
		\begin{tikzcd}
			N\ar[r]\ar[d,"f"] & \varphi_*N \ar[d,"\varphi_*f"]\\
			P\ar[r]           & \varphi_*P
		\end{tikzcd}
	\end{center}
	commutes, where the horizontal maps are the pushout maps.
\end{lemma}
\begin{proof}
	It is enough to apply the universal property of the pushout of
	$\varphi_*N$ to the diagram
	\begin{center}
		\begin{tikzcd}
		L \ar[d,swap,"i"]\ar[r,"\varphi"] &
			M\ar[d,"\varphi_*i"]\ar[ddr,bend left,"\varphi_*j"]\\
		N \ar[dr, bend right,swap,"f"]\ar[r] &
			\varphi_*N \ar[dr,dashed,"\varphi_*f"]\\
		& P \ar[r] & \varphi_*P 
		\end{tikzcd}
	\end{center}
	Indeed the map $\varphi_*f:\varphi_*N\to\varphi_*P$, whose existence is
	ensured by the universal property, is such that
	$\varphi_*P/\varphi_*f(\varphi_*N)\cong P/f(N)$ is $J$-torsion.
\end{proof}

\begin{proposition} \label{prop:pushforwardFunctor}
	Let $\varphi:L\to M$ be an injective and pure map of \pointed{} $R$-modules.
	The diagram
	\begin{center}
		\begin{tikzcd}
			(N,i,t) \arrow[r,mapsto] \arrow[d,"f"] &
				(\varphi_*N,\varphi_*i,\varphi_*t) \ar[d,"\varphi_*f"]\\
			(P,j,u) \arrow[r,mapsto] & (\varphi_*P, \varphi_*j,\varphi_*u)
		\end{tikzcd}
	\end{center}
	where $\varphi_*f$ is as in Lemma \ref{lemma:pushforward2},
	defines a functor from $\jtcat{(L,r)}$ to $\jtcat{(M,s)}$.
\end{proposition}
\begin{proof}
	In view of Lemmas \ref{lemma:pushforward1} and 
	\ref{lemma:pushforward2} it is enough to show that $\varphi_*$
	behaves well with respect to
	the composition of maps of \jtextensions{}. This is
	immediate from the construction in Lemma \ref{lemma:pushforward2} and
	the uniqueness part of the universal property of the pushout.
\end{proof}

\begin{definition}
	We call the functor of Proposition \ref{prop:pushforwardFunctor} the
	\emph{pushforward along $\varphi$}, and we denote it by $\varphi_*$.
\end{definition}

\begin{theorem}
	Let $\varphi:(L,r)\into (M,s)$ be an injective pure map of \pointed{}
	$R$-modules. Then the functor $\varphi_*$ is left adjoint to $\varphi^*$.
\end{theorem}
\begin{proof}
	Since $\varphi$ is injective we will, for simplicity,
	denote $\varphi(L)$ by $L$.

	Let $(N,i,t)$ be a \jtextension{} of $L$ and let $(P,j,u)$ be a
	\jtextension{} of $M$. We want to show that we have
	\[ \Hom_{\jtcat{(L,r)}}(N,\varphi^*P) \cong
	   \Hom_{\jtcat{(M,s)}}(\varphi_*N,P) \]
	naturally in $N$ and $P$.

	Let $f:N\to \varphi^*P$ be a map of \jtextensions{} of $L$; notice that
	in particular $f\circ i=\varphi^*j$. Composing $f$ with the natural
	inclusion $\varphi^*P\into P$ we get a map of \pointed{} $R$-modules
	$f':N\to P$ such that $f'\circ i =j\circ \varphi$, so by the universal
	property of the pushout there exists a unique map 
	$g:\varphi_*N\to P$ that is a map of \jtextensions{} of $M$.

	We define a map
	\[ \Psi_{N,P}: \Hom_{\jtcat{(L,r)}}(N,\varphi^*P) \to
	   \Hom_{\jtcat{(M,s)}}(\varphi_*N,P) \]
	by letting $\Psi_{N,P}(f):=g$. The map $\Psi$ is natural in $N$ and
	$P$, since it is defined by means of a universal property.
	Indeed, if $h:N'\to N$ is a map of \jtextensions{} of $L$ and $f'=f\circ h$
	then $\Psi_{N',P}(f')$ is by definition the unique map $\varphi_*N'\to P$
	that makes the pushout diagram commute so it must coincide with
	$g\circ \varphi_*h$. Similarly if $k:P\to P'$ is a map of \jtextensions{}
	of $M$ then $\Psi_{N,P'}(\varphi^*k\circ f)$ must coincide with
	$k\circ g$.

	To see that the map $\Psi_{N,P}$ is injective, let $f':N\to \varphi^*P$
	be another map and assume that $\Psi_{N,P}(f)=\Psi_{N,P}(f')$. But then
	the composition of $\Psi_{N,P}(f)$ with the pushout map $N\to\varphi_*N$
	coincides with the composition of $f$ and the natural inclusion
	$\varphi^*P\into P$, and analogously for $f'$, so we conclude that
	$f=f'$.

	To see that $\Psi_{N,P}$ is surjective, let $g':\varphi_*N\to P$ be a
	map of \jtextensions{} of $M$. Then by definition of pullback its
	composition with $N\to \varphi_*N$ factors through $\varphi^*P\into P$
	as a map of \jtextensions{} $f':N\to\varphi^*P$, and again by the
	uniqueness of the map of the universal property of the pushout one
	can check that $\Psi_{N,P}(f')=g'$.
\end{proof}

\begin{remark}
\label{rem:unitCounit}
	Let $\varphi:L\into M$ be an injective and pure map of \pointed{} $R$-modules
	and let $(N,i,t)$ and $(P,j,u)$ be \jtextensions{} of $L$ and $M$
	respectively. We can give an explicit description of the unit
		      \[\eta_N:N\to\varphi^*\varphi_*N\]
	and the counit
		      \[\varepsilon_P:\varphi_*\varphi^*P\to P\]
	of the adjunction.

	Notice that the pushout map $N\to \varphi_*N$ is injective. Moreover, since
	$N$ is a \jextension{} of $L$, the image of this map is contained in
	$\varphi^*\varphi_*N=\divm{J}{\varphi_*i(\varphi(L))}{\varphi_*N}$.
	The resulting inclusion $N\into \varphi^*\varphi_*N$ is the unit $\eta_N$.

	By definition $\varphi^*P$ is contained in $P$, and the diagram
	\begin{center}
		\begin{tikzcd}
			L \ar[r,hook,"\varphi"] \ar[d] & M \ar[d,"j"] \\
			\varphi^*P \ar[r,hook] & P
		\end{tikzcd}
	\end{center}
	commutes, so by the universal property of the pushout there exists
	a map $\varphi_*\varphi^*P\to P$. This map is the counit $\varepsilon_P$.
\end{remark}

The following examples of pullback and pushforward functors are of particular
importance to us, because they will be key to the construction of
maximal \jtextensions{}.

\begin{definition}
	Let $M$ be a \pointed{} $R$-module and let $\inct M:M[J]\to M$ be the
	natural inclusion of its torsion submodule. We will call the pullback functor
	$\inct M^*$ the \emph{torsion} functor and we will denote it by $\jt$.
\end{definition}

\begin{remark} \label{rem:torsUnitIsIsom}
	For every \jtextension{} of $\jt(M)$ the unit map
	$\eta_N:\jt ((\inct M)_*N)\to N$ is an isomorphism. Indeed, we have
	$\jt((\inct M)_*N)=((\inct M)_*N)[J]=N[J]$, and since $N$ is a
	\jtextension{} of a $J$-torsion module and $J$ is complete then $N[J]=N$.
\end{remark}

Notice that the inclusion $\incs M$ of a \pointed{} $R$-module into its
saturation is injective and pure.

\begin{definition}
	Let $M$ be a \pointed{} $R$-module and let $\incs M:M\to\tp(M)$ be the
	inclusion into its saturation. We will call the pushforward functor
	$(\incs M)_*$ the \emph{saturation} functor and we will denote it by $\tp$.
\end{definition}

\begin{remark} \label{rem:satCounitIsIso}
	The counit map $\varepsilon_P:P\to\tp(\incs{M}^*P)$ is an isomorphism.
	Indeed, one can see from the definition of pullback that $\incs M^*P=P$
	is saturated, hence it coincides with its own saturation.
\end{remark}

\subsection{Maximal \jtextensions{}}
\label{sec:maxJText}

Maximal \jtextensions{} are the analogue of algebraic closures in field
theory. The main result of this section is the proof of the existence
of a maximal \jtextension{} for any \pointed{} $R$-module, and we achieve
this by first constructing such an extension for its torsion and its saturation.

\begin{definition}
	A \jtextension{} $\Gamma$ of the \pointed{} $R$-module $M$ is called
	\emph{maximal} if for every \jtextension{} $N$ of $M$ there is a map of
	\jtextensions{} $\varphi:N\into\Gamma$.
\end{definition}
 
The definition of \pointed{} $R$-module already provides a maximal
\jtextension{} for any $J$-torsion module.

\begin{lemma} \label{lemma:maxJTextTorsion}
	Let $(M,s)$ be a \pointed{} $R$-module. If $M$ is $J$-torsion, then
	$(T,s,\id_T)$ is a maximal \jtextension{} of $(M,s)$.
\end{lemma}
\begin{proof}
	If $(N,i,t)$ is a \jtextension{} of $M$, then in particular we have
	\begin{align*}
		N=\divm {J}{i(M)}N=\divm J{\divm J0{i(M)}}N\subseteq
			\divm J{\divm J0N}N=\divm J0N=N[J]
	\end{align*}
	so $N$ is $J$-torsion. Then $t:N\into T$ satisfies $t\circ i=s$
	and $\id_T\circ t=t$, so it is a map of \jtextensions{}.
\end{proof}

The existence of a maximal \jtextension{} of a saturated module comes from
the existence of a \jhull{}, and it requires only a little more technical work.

\begin{lemma} \label{lemma:maxJTextPush}
	Let $(M,s)$ be a saturated \pointed{} $R$-module
	and let $\iota:M\into \Gamma$ be a \jhull{} of $M$. Then
	\begin{enumerate}
		\item $\restr\iota{M[J]}:M[J]\into \Gamma[J]$ is an isomorphism.
		\item $(\Gamma,\iota,\tau)$ is a maximal \jtextension{} of $(M,s)$,
		      where $\tau:=s\circ \restr\iota{M[J]}^{-1}$. 
	\end{enumerate}
\end{lemma}

\begin{proof}
	For (1) notice that $\restr\iota{M[J]}:M[J]\into \Gamma[J]$ is an essential 
	extension by Lemma \ref{lemma:Jessential}, so it is an isomorphism by 
	Proposition \ref{proposition:EssentialExtensionOfInjective}.

	For (2) we have that $\Gamma$ is a \jtextension{} of $M$, because it is a
	\jextension{} and $\tau\circ\restr\iota{M[J]}=s$. Let $(N,i,t)$ be any
	\jtextension{} of $M$. Since $i:M\into N$ is a \jextension{}, there is a
	homomorphism $\varphi:N\to\Gamma$ such that $\varphi\circ i=\iota$.
	Moreover, since $t\circ\restr i{M[J]}=s$ and 
	$\tau\circ\restr{(\varphi\circ i)}{M[J]}= \tau\circ\restr\iota{[M[J]}=s$, 
	we have $\tau\circ\restr\varphi{N[J]}=t$, so $\varphi$ is a map of 
	\jtextensions{}. It follows that $\Gamma$ is a maximal \jtextension{} of $M$.
\end{proof}

Finally we can construct a \jtextension{} of any \pointed{} $R$-module.

\begin{proposition} \label{proposition:maxJTextEquivalencePT}
	Let $(\Gamma,\iota,\tau)$ be a \jtextension{} of the \pointed{} $R$-module
	$(M,s)$ such that $\Gamma$ is saturated.
	Then $\Gamma$ is a maximal \jtextension{} of $M$ if and only if
	$\tp(\Gamma)$ is a maximal \jtextension{} of $\tp(M)$.
\end{proposition}
\begin{proof}
	Assume first that $\Gamma$ is a maximal \jtextension{} of $M$ and let
	$(N,i,t)$ be a \jtextension{} of $\tp(M)$.
	Then there is a map $\varphi:\incs M^*N\to \Gamma$ of \jtextensions{} of
	$M$, so there is a map $\tp(\varphi):\tp(\incs M^*N)\to \tp(\Gamma)$
	of \jtextensions{} of $\tp(M)$. By Remark \ref{rem:satCounitIsIso}
	we have $N\cong\tp(\incs M^*N)$, so there is also a map $N\to
	\tp(\Gamma)$.  This proves that $\tp(\Gamma,\iota,\tau)$ is a maximal
	\jtextension{} of $\tp(M)$.

	Assume now that $\tp(\Gamma)$ is a maximal \jtextension{} of $\tp(M)$.
	Let $(N,i,t)$ be a \jtextension{} of $M$. Then there is a map of
	\jtextensions{} $f:\tp(N)\to\tp(\Gamma)$
	completing the following diagram:
	\begin{center}
		\begin{tikzcd}
			M[J]\arrow[rrr,hook]\arrow[ddd,hook,"s"]
			\arrow[drr,hook,"\restr i{M[J]}"] & & &
			M\arrow[ddd,hook,"\incs M"]\arrow[drr,hook,"i"]
			\arrow[rrrrrrd,hook,"\iota"]\\
			& & N[J] \arrow[rrr, hook]\arrow[ddd, hook, "t"] & & &
			N \arrow[ddd, hook, "\incs N"] \arrow[rrrr,dashed,"\varphi"]& & & &
			\Gamma\arrow[ddd,hook,"\incs \Gamma"]\\ \\
			T \arrow[rrr, hook, "\tp(s)^{-1}"]
				\arrow[drr, hook, "\id_T"] & & &
			\tp(M) \arrow[drr, hook,swap, "\tp(i)"]
				\arrow[drrrrrr,hook,"\tp(\iota)"] \\
			& & T \arrow[hook, rrr, swap, "\tp(t)^{-1}"] & & &
				\tp(N)\arrow[rrrr,swap,hook,"f"]
			& & & & \tp(\Gamma)
		\end{tikzcd}
	\end{center}
	Notice that since $\Gamma$ is saturated the map
	$\incs \Gamma:\Gamma\into\tp(\Gamma)$ is an isomorphism. So we can define 
	$\varphi:=\incs \Gamma^{-1}\circ f\circ \incs N:N\to \Gamma$ and we have
	\begin{align*}
		\incs \Gamma\circ\varphi\circ i =f\circ \incs N\circ i=
		f\circ \tp(i)\circ \incs M= \tp(\iota)\circ \incs s =\incs \Gamma
		\circ \iota
	\end{align*}
	hence $\varphi\circ i=\iota$. Moreover, since
	$\tp(\tau)\circ\incs \Gamma=\tau$, we have
	\begin{align*}
		\tau\circ\restr\varphi{N[J]}
		&=\tau\circ\incs \Gamma^{-1}\circ f\circ\restr {\incs N}{N[J]}= \\
		&=\tau\circ\incs \Gamma^{-1}\circ f\circ \tp(t)^{-1}\circ t= \\
		&=\tau\circ\incs \Gamma^{-1}\circ\tp(\tau)^{-1}\circ t=\\&=t
	\end{align*}
	so $\varphi$ is a map of \jtextensions{}.
	Hence $\Gamma$ is a maximal \jtextension{} of $M$.
\end{proof}

\begin{theorem} \label{theorem:existenceUniquenessMaximalJText}
	Every \pointed{} $R$-module $M$ admits a maximal \jtextension{}.
	Moreover, for any maximal \jtextension{} $\Gamma$ of $M$ the following hold:
	\begin{enumerate}
		\item If $\Gamma'$ is another maximal \jtextension{} of $M$, then
		      $\Gamma\cong\Gamma'$ as \jtextensions{};
		\item The module $\Gamma$ is saturated;
		\item The module $\Gamma$ is \jinj{};
		\item If $(N,i,t)$ is a \jtextension{} of $M$ and $\varphi:N\to\Gamma$
		      is a map of \jtextensions{}, then $(\Gamma,\varphi,\tau)$ is a 
			  maximal \jtextension{} of $(N,t)$.
	\end{enumerate}
\end{theorem}

\begin{proof}
	Let $j:\tp(M)\into\Gamma$ be a \jhull{} of the \push{} of $M$ and 
	let $\tau:=\tp(s)\circ\restr j{\tp(M)[J]}^{-1}$. By Lemma 
	\ref{lemma:maxJTextPush} we have that $(\Gamma,j,\tau)$ is a maximal 
	\jtextension{} of $\tp(M)$.
	By Remark \ref{rem:satCounitIsIso} we have that $(\Gamma,\iota,\tau)=
	\inct M^*(\Gamma,j,\tau)$ is a \jtextension{} of $M$ such that
	$\tp(\Gamma,\iota,\tau)\cong (\Gamma,j,\tau)$, so by
	Proposition \ref{proposition:maxJTextEquivalencePT} we conclude that
	it is a maximal \jtextension{} of $M$.

	Let now $(\Gamma',\iota',\tau')$ be another maximal \jtextension{} of 
	$(M,s)$. Then there is a map of \jtextensions{} $f:\Gamma\into \Gamma'$ 
	which is an essential $J$-extension by Lemma
	\ref{lemma:Jessential}, as it is an isomorphism on the $J$-torsion.
	Since $\Gamma$ is \jinj{} we have that $f$ is an 
	isomorphism by Proposition \ref{proposition:EssentialExtensionOfInjective}. 
	This shows that any maximal \jtextension{} of $M$ is isomorphic to 
	$\Gamma$, which proves (1), (2) and (3) at once.

	For (4) it is enough to notice that if $j:\tp(M)\into \Gamma$ is a 
	\jhull{}, then so is $\tp(\varphi)$, thus by the same
	argument as above $\Gamma$ is a maximal \jtextension{} of $N$.
\end{proof}


\section{Automorphisms of \jtextensions{}}
\label{section:aut}

Fix for this section a ring $R$, a complete ideal filter $J$ of $R$
and a $J$-torsion and \jinj{} left $R$-module $T$.
Fix moreover a \pointed{} $R$-module $(M,s)$ and a maximal \jtextension{}
$(\Gamma,\iota,\tau)$ of $(M,s)$.

\subsection{Normal extensions}

We define normal extensions in analogy with field theory.

\begin{definition}
	A \jtextension{} $i:M\into N$ is called
	\emph{normal} if every injective $J$-map
	$f:N\into \Gamma$ such that $f\circ i=\iota$ has the same image.
\end{definition}

Notice that we are considering all injective $J$-maps that respect
$\iota:M\into \Gamma$, even if they are not maps of \jtextensions{},
that is even if they do not respect the embeddings of the torsion
submodules into $T$.

\begin{remark}
	Although we will not make use of it, it interesting to notice that
	the group $\Aut_M(N)$ acts on $\Emb_M(N,\Gamma)$ by composition
	on the right. It is then easy to see that $N$ is normal if and
	only if this action is transitive.

	This is reminiscent of Galois theory \emph{à la Grothendieck}. One
	might wonder if, assuming the necessary finiteness conditions on
	automorphism groups hold, the category of \jtextensions{}
	is indeed a Galois category with fundamental functor
	$\Emb_M(-,\Gamma)$. Unfortunately, the fact that in general pushouts
	of \jtextensions{} do not exist (see Remark \ref{remark:noPushout})
	implies that this is not the case.

	We may refine this question as follows: does the
	category of \jtextensions{} embed as the subcategory of connected
	objects of some Galois category?
\end{remark}

\begin{proposition}
	Every saturated \jtextension{} of $M$ is normal.
\end{proposition}
\begin{proof}
	Assume that $M$ is saturated, let $i:M\into N$ be a \jtextension{} and
	let $f,g:N\into \Gamma$ be injective $J$-maps with
	$f\circ i=g\circ i=\iota$.
	If $f(N)\neq g(N)$, we may assume without loss of generality that there
	is $n\in N$ with $f(n)\not \in g(N)$. Then $t:=f(n)-g(n)\in \Gamma[J]$
	by Lemma \ref{lemma:diffIsTorsion}. Since $N$ is saturated and $g$ is
	injective we have $t\in g(N)$, thus $f(n)=g(n)+t\in g(N)$, a
	contradiction. We deduce that $f(N)=g(N)$, so $N$ is normal.
\end{proof}

\begin{corollary}
	Every maximal \jtextension{} is normal.
\end{corollary}

\subsection{A fundamental exact sequence}

\begin{proposition} \label{prop:homFixTorsIsSat}
	Let $(N,i,t)$ be a normal \jtextension{} of $(M,s)$ 
	and let $\Aut_{M+N[J]}(N)$ denote the subgroup of $\Aut_M(N)$
	consisting of those automorphisms that restrict to the identity on
	the submodule of $N$ generated by $i(M)$ and $N[J]$.
	Then the restriction map along $\incs N:N\to \tp(N)$
	\begin{align*}
		\Aut_{\tp(M)}(\tp(N)) \to \Aut_{M+N[J]}(N)
	\end{align*}
	is a well-defined group isomorphism.
\end{proposition}
\begin{proof}
	Let us identify for simplicity $N$ with its image $\incs N(N)$ in $\tp(N)$,
	and let
	$\sigma\in\Aut_{\tp(M)}({\tp(N)})$. To see that the image
	of $\restr\sigma N$ is contained in $N$, let $f:\tp(N)\into \Gamma$ be a
	map of \jtextensions{} of $\tp(M)$, which is necessarily also a map
	of \jtextensions{} of $M$. Since $\tp(s)$ is an isomorphism, also
	$f\circ \sigma$ is a map of \jtextensions{} of $\tp(M)$, and since $N$
	is normal we have that the image of $N$ in $\Gamma$ under $f$ and under
	$f\circ\sigma$ are the same, which shows that $\sigma(N)=N$. Since this
	holds for both $\sigma$ and its inverse, we have that $\restr\sigma N\in
	\Aut_M(N)$, and clearly $\sigma$ is the indentity on $N[J]$.

	To show that the restriction to $N$ is an isomorphism, we construct an
	inverse. Let now $\sigma\in\Aut_{M+N[J]}(N)$, and recall that we can see it
	as a map of \jtextensions{} of $(M,s)$
	\begin{align*}
		\sigma: (N,t) \to (N, t\circ\restr{\sigma}{N[J]})\,.
	\end{align*}
	Composing it with $\incs N$ we get a map
	\begin{align*}
		\incs N\circ \sigma :(N,t) \to
			(\tp(N), (\incs N)_*(t\circ \restr{\sigma}{N[J]}))\,.
	\end{align*}
	Moreover, the map $\tp (i)$ is also a map of \jtextensions{}
	\begin{align*}
		\tp(i): (\tp(M),(\incs M)_*s) \to
			(\tp(N), (\incs N)_*(t\circ \restr{\sigma}{N[J]}))
	\end{align*}
	so by the universal property of the pushout there is a map of
	\jtextensions{}
	\begin{align*}
		\sigma':(\tp(N),(\incs N)_*t), 
			(\tp(N), (\incs N)_*(t\circ \restr{\sigma}{N[J]}))\,.
	\end{align*}
	It is straightforward to check that $\sigma\mapsto \sigma'$ provides
	an inverse for the restriction map $\Aut_{\tp(M)}(\tp(N))\to\Aut_M(N)$,
	which is then an isomorphism.
\end{proof}

\begin{proposition}
\label{prop:AutIsHom}
	Let $(N,i,t)$ be a \jtextension{} of $(M,s)$.  Then the map
	\begin{align*}
		\varphi: \Aut_{M+N[J]}(N) &\to
		\Hom\left(\frac{N}{i(M)+N[J]}, N[J]\right) \\
		\sigma & \mapsto \left(\varphi_\sigma:[n]\mapsto
		\sigma(n)-n\right)
	\end{align*}
	is an isomorphism of groups. In particular, $\Aut_{\tp(M)}(\tp(N))$ is abelian.
\end{proposition}
\begin{proof}
	We will denote by $[n]$ the class of an element $n\in N$ in
	$N/(i(M)+N[J])$.
	Notice that for any $\sigma\in \Aut_{M+N[J]}(N)$ we have
	$\sigma(n)-n\in N[J]$ by Lemma \ref{lemma:diffIsTorsion}, and
	$\varphi_\sigma$ is a homomorphism of $R$-modules. 
	To see that $\sigma\mapsto \varphi_\sigma$ is a group homomorphism, let
	$\sigma'\in\Aut_{M+N[J]}(N)$.
	Then, since $\sigma$ is the identity on $N[J]$ and $\sigma'(n)-n\in N[J]$,
	we have
	\begin{align*}
		\sigma(\sigma'(n))-n
		&= \sigma(\sigma'(n))-n +\sigma'(n)-n -\sigma(\sigma'(n)-n) \\
		&= \sigma(n)-n + \sigma'(n)-n
	\end{align*}
	which shows that $\varphi$ is a group homomorphism. It is also clearly
	injective, because if $\varphi_\sigma(n)=n$ then $\sigma$ must be the
	identity.

	To prove surjectivity it is enough to show that for any $R$-module
	homomorphism $h:N/(i(M)+N[J])\to N[J]$ the map
	\begin{align*}
		\sigma_h: N&\to N\\
		n&\mapsto n+h([n])
	\end{align*}
	which is clearly the identity on $i(M)+N[J]$, is an automorphism of $N$.
	It is injective, because if $n=-h([n])$ then in particular
	$n$ is torsion and thus $[n]=0$.
	It is also surjective, because for any $n\in N$ we have
	\begin{align*}
		\sigma_h(n-h([n])) 
		&= n-h([n]) + h([n-h([n])]) \\
		&= n- h([n] - [n+h([n])])\\
		&=n
	\end{align*}
\end{proof}

\begin{corollary}
	\label{cor:homSatQuotient}
	Let $(N,i,t)$ be a normal \jtextension{} of $M$. Denoting for
	simplicity by $\tp(M)$ the image of $\tp(M)$ inside $\tp(N)$ we have
	\begin{align*}
		\Aut_{\tp(M)}(\tp(N)) \cong
		\Hom\left(\frac{\tp(N)}{\tp(M)}, \jt(N)\right)\,.
	\end{align*}
\end{corollary}
\begin{proof}
	The claim follows from the two propositions above and the
	fact that
	\begin{align*}
		\frac{N}{i(M)+N[J]}\cong \frac{\tp(N)}{\tp(M)}\,.
	\end{align*}
	To see that the two quotients are isomorphic, consider the map
	$N\to \tp(N)/\tp(M)$ given by $n\mapsto \incs N(n) + \tp(M)$:
	its kernel is $i(M)+N[J]$ and it is surjective because
	$\tp(N)$ is generated by the images of $N$ and $T$.
\end{proof}

\begin{remark}
	Let $N$ be a \jtextension{} of $M$ and let $\sigma \in \Aut_M(N)$. The
	restriction of $\sigma$ to $N[J]$ is an element of $\Aut_{M[J]}(N[J])$.
	Indeed, the image of a $J$-torsion element under a map of \jtextensions{}
	is again a $J$-torsion element; since this is true for both $\sigma$ and
	$\sigma^{-1}$ we can conclude that $\restr{\sigma}{N[J]}:N[J]\to N[J]$ is
	an automorphism.
\end{remark}

\begin{lemma} \label{lemma:restr_surj}
	If $(N,i,t)$ is a normal \jtextension{} of $(M,s)$, the restriction map
	\[\Aut_M(N)\to \Aut_{M[J]}(N[J])\] is surjective.
\end{lemma}
\begin{proof}
	Let $\sigma\in \Aut_{M[J]}(N[J])$. Notice that $(N,i,t\circ\sigma)$ is
	also a \jtextension{} of $M$, and let $f:(N,i,t)\into(\Gamma,\iota,\tau)$
	and $g:(N,i,t\circ\sigma)\into(\Gamma,\iota,\tau)$ be maps of
	\jtextensions{}. Since $N$ is normal we have $f(N)=g(N)$, thus
	$f^{-1}\circ g$ is an automorphism of $N$ that restricts to $\sigma$.
\end{proof}

The exact sequence appearing in the following theorem has been studied, in some
particular cases, in \cite{abtien}, \cite{palenstijn} and \cite{mypaper}.

\begin{theorem}
\label{thm:exactSequence}
	Let $M$ be a \pointed{} $R$-module and let $N$ be a normal \jtextension{}
	of $M$.  Then there is an exact sequence of groups
	\begin{align*}
		1\to \Hom\left(\frac{\tp(N)}{\tp(M)}, \jt(N)\right) \to
			\Aut_M(N) \to \Aut_{\jt(M)}(\jt(N))\to 1
	\end{align*}
	Moreover 
	$\Aut_{\jt(M)}(\jt(N))$ acts on $\Hom(\tp(N)/\tp(M),\jt(N))$ by composition.
\end{theorem}
\begin{proof}
	By Lemma \ref{lemma:restr_surj} the map $\Aut_M(N)\to\Aut_{\jt(M)}(\jt(N))$
	is surjective and its kernel is $\Aut_{i(M)+N[J]}(N)$ by definition.
	By Proposition \ref{prop:homFixTorsIsSat} this group is isomorphic to
	$\Aut_{\tp(M)}(\tp(N))$ via the restriction under $\incs N: N\to\tp(N)$.
	Combining this with Corollary \ref{cor:homSatQuotient}
	we get the desired exact sequence.

	The fact that $\Aut_{\jt(M)}(\jt(N))$ acts on $\Aut_{i(M)+N[J]}$ by
	conjugation is a standard result on short exact sequences with abelian
	kernel, and one can trace this action under the isomorphisms described
	above to check that on $\Hom(\tp(N)/\tp(M),\jt(N))$ 
	this action is indeed the
	composition of maps, similarly to \cite[Proposition 3.12]{mypaper}.
\end{proof}

\section{Kummer theory for algebraic groups}
\label{section:kummer}

\subsection{General theory}

Let $K$ be a field and fix a separable closure $K_s$ of $K$.
Let $G$ be a commutative algebraic group over $K$, let $R\subseteq \End_K(G)$
be a subring of the ring of $K$-endomorphisms of $G$ and let
$M\subseteq G(K)$ be an $R$-submodule.
Let $J$ be a complete ideal filter of $R$, let
$T:=G(\Kbar)[J]$ and let $\Gamma:=\divm JM{G(\Kbar)}$.

We are interested in studying the field extension $K(\Gamma)$ of $K$,
that is the fixed field of the subgroup of $\Gal(K_s\mid K)$ that acts
trivially on $\Gamma$, and we want to do so using the theory of
\jtextensions{} introduced in the previous section.
A necesary and sufficient condition in order to proceed this way is
that $T=G(\Kbar)[J]$ be $J$-injective: indeed in this case
$\Gamma$ is a saturated, and thus normal, \jtextension{} of $M$.

\begin{remark}
	\label{remark:maximalorders}
	The condition that $T$ is $J$-injective for some, and in fact for all,
	ideal filters $J$, holds 
	for example if $G$ is a simple abelian variety with $R$ a maximal
	order in the division algebra $\End_{\Kbar}(G)\otimes \Q$.
	Indeed in this case every non-zero element
	$r$ of $R$ is surjective on $G(\Kbar)$, which implies that $T$ is
	divisible: if an element $u\in G(\Kbar)$ is such that $ru=t\in T$ and
	$I\in J$ is such that $It=0$, then since $I$ is a right ideal we
	have $Iu=0$, so $u\in T$; hence $r:T\to T$ is surjective and $T$ is
	divisible.

	It follows that $T$ is injective: this is a well-known statement if
	$R$ is a Dedekind domain, but the proof can be adapted
	to the non-commutative case as follows.
	Let $I$ be a left ideal of $R$ and let $f:I\to T$ be a map that we
	wish to extend to a map $\tilde f:R\to T$.
	By \cite[Theorem 22.7]{reiner} there is a right \emph{fractional}
	ideal $J$ of $R$ such that $IJ=R$ and $1\in JI\subseteq R$.
	In particular
	there are non-zero elements $b_1,\dots,b_n\in J$ and $a_1,\dots, a_n\in I$
	such that $\sum_{i=1}^nb_ia_i=1$, and since $T$ is divisible there are
	$x_1,\dots, x_n\in T$ such that $a_ix_i=f(a_i)$. It follows that for every
	$y\in I$ we have
	\begin{align*}
		f(y) = f\left(y\sum_{i=1}^nb_ia_i\right) = 
		       \sum_{i=1}^n(yb_i)f(a_i) = y\sum_{i=1}^n(b_ia_i)x_i
	\end{align*}
	and we can let $\tilde f(r)=r\sum_{i=1}^n(b_ia_i)x_i$ for
	every $r\in R$.
\end{remark}

Let us then assume that $T=G(\Kbar)[J]$ is $J$-injective,
so that $\Gamma$ is a saturated, therefore normal, \jtextension{} of $M$.
Then the standard exact sequence of groups
coming from the tower of Galois extensions $K\subseteq K(T) \subseteq
K(\Gamma)$ maps into the exact sequence \ref{thm:exactSequence} via the
Galois action on the points of $G$, and we obtain the following
commutative diagram of groups with exact rows:
\begin{equation*} \begin{tikzcd}
	1 \ar[r] & \Gal(K(\Gamma)\mid K(T)) \ar[d, hook, "\kappa"] \ar[r]
	         & \Gal(K(\Gamma)\mid K)    \ar[d, hook, "\rho"]   \ar[r]
			 & \Gal(K(T)\mid K)         \ar[d, hook, "\tau"]   \ar[r]
			 & 1 \\
	1 \ar[r] & \Hom\left(\frac{\Gamma}{\tp(M)}, T\right) \ar[r]
			 & \Aut_M(\Gamma) \ar[r]
			 & \Aut_{\jt(M)}(T) \ar[r]
			 & 1
\end{tikzcd} \end{equation*}

Notice that the action of $\Aut_{M[J]}(T)$ on $\Hom(\Gamma/(M+T),T)$
restricts to an action of $\im(\tau)$ on $\im(\kappa)$.

\begin{definition}
	In the situation described above we will call the maps $\kappa$, $\tau$
	and $\rho$ the \emph{Kummer representation}, the
	\emph{torsion representation} and the \emph{torsion-Kummer representation},
	respectively.
\end{definition}

As in Section \ref{section:duality}, if $N$ and $P$ are $R$-modules and
$S$ is a subset of $\Hom_R(N,P)$ we let $\ker(S)=\bigcap_{f\in S}\ker(f)$.

\begin{theorem}
	\label{thm:sesCohomology}
	There is an exact sequence of abelian groups
	\begin{align*}
		0\to \frac{\divm{J}{\tp(M)}{\tp(G(K))}}{\tp(M)} \to
			 \ker(\im(\kappa)) \to
			 H^1(\im(\tau),T)
	\end{align*}
\end{theorem}
\begin{proof}
	By Lemma \ref{lemma:diffIsTorsion} for any $b\in G(K(T))$ we may define a map
	\begin{align*}
		\varphi_b:\im(\kappa)&\to T\\
		\sigma & \mapsto \sigma(b)-b
	\end{align*}
	which is a cocycle. It follows that the map
	\begin{align*}
		\varphi: G(K(T)) & \to H^1(\im(\tau),T)\\
		b & \mapsto \varphi_b
	\end{align*}
	is a group homomorphism. Moreover its kernel is
	\begin{align*}
		\ker(\varphi)
			&=\set{b\in G(K(T))\mid\varphi_b\text{ is a coboundary}}\\
			&=\set{b\in G(K(T))\mid\exists\,t\in T\text{ such that }
				\sigma(b)-b=\sigma(t)-t\,\forall\,\sigma\in\im(\kappa)}\\
			&=\set{b\in G(K(T))\mid\exists\,t\in T\text{ such that }
				\sigma(b-t)=b-t\,\forall\,\sigma\in\im(\kappa)}\\
			&=G(K)+T
	\end{align*}
	so that we have an exact sequence
	\begin{align*}
		0\to G(K)+T\to G(K(T))\to H^1(\im(\tau),T)
	\end{align*}
	and considering the intersection of the first two terms with $\Gamma$ we get
	\begin{align*}
		0\to \Gamma\cap(G(K)+T)\to\Gamma\cap G(K(T))\to H^1(\im(\tau),T)\,.
	\end{align*}
	Since $M+T\subseteq \Gamma\cap(G(K)+T)$ we also have
	\begin{align*}
		0\to \frac{\Gamma\cap(G(K)+T)}{M+T}\to
		     \frac{\Gamma\cap G(K(T))}{M+T}\to H^1(\im(\tau),T)\,.
	\end{align*}
	Rewriting $M+T=\tp(M)$ and $G(K)+T=\tp(G(K))$, noticing that
	\[\Gamma\cap\tp(G(K))=\divm{J}{\tp(M)}{\tp(G(K))}\] and that
	\begin{align*}
		\ker(\im(\kappa))
			&=\set{x\in\frac{\Gamma}{M+T}\mid f(x)=0\,\forall\,f\in\im(\kappa)}\\
			&=\frac{\set{\tilde x\in\Gamma\mid
				\sigma(\tilde x)=\tilde x\,\forall\,\sigma\in\im(\kappa)}}{M+T}\\
			&=\frac{\Gamma\cap G(K(T))}{M+T}
	\end{align*}
	we get the desired exact sequence.
\end{proof}

The following theorem generalizes \cite[Theorem 5.9]{mypaper}.
\begin{theorem}
	\label{thm:main}
	Assume that the $\End(T)$-submodule of $\Hom(\Gamma/\tp(M),T)$
	generated by $\im(\kappa)$ is finitely generated. If
	the following three conditions hold
	\begin{enumerate}
		\item There is a positive integer $d$ such that
		      \[ d\cdot\divm{J}{\tp(M)}{\tp(G(K))} \subseteq \tp(M) \]
		\item There is a positive integer $n$ such that
		      \[ n\cdot H^1(\im(\tau),T) = 0 \]
		\item There is a positive integer $m$ such that the subring
		      of $\End(T)$ generated by $\im(\tau)$ contains
		      \[ m\cdot \End(T) \]
	\end{enumerate}
	then $\im(\kappa)$ contains $dnm\cdot \Hom(\Gamma/\tp(M),T)$.
\end{theorem}
\begin{proof}
	Let $V$ be the $\End(T)$-submodule of $\Hom(\Gamma/\tp(M), T)$
	generated by $\im(\kappa)$ and let $X=\Gamma/\tp(M)$. From (1) and (2)
	it follows that $\ker(V)=\ker(\im\kappa)\subseteq X[dn]$.
	Since $V$ is finitely generated as an $\End(T)$-module, by Proposition
	\ref{prop:duality1} we have
	\begin{align*}
		V = \Hom\left(\frac{X}{\ker(V)}, T\right) \supseteq
		    \Hom\left(\frac{X}{X[dn]}, T\right) \supseteq
		    dn\cdot \Hom(X,T)\,.
	\end{align*}
	Since $\im(\kappa)$ is an $\im(\tau)$-module, we have
	\begin{align*}
		\im(\kappa) = \im(\tau)\cdot \im(\kappa) \supseteq
		m\cdot \End(T)\cdot \im(\kappa) = m\cdot V \supseteq
		dnm\cdot\Hom(X,T)
	\end{align*}
	and we conclude.
\end{proof}

\subsection{Elliptic curves over number fields}

We keep the notation of the previous section and we further assume
that $K$ is a number field, that $G=E$ is an elliptic curve and
that $R=\End_K(E)$.
In particular we have that $K_s=\Kbar$ and that $R$ is either $\Z$ or an
order in an imaginary quadratic number field.
Up to replacing $K$ by an extension of degree
$2$ we may assume that $\End_K(E)=\End_{\Kbar}(E)$.

Notice that $T=E(\Kbar)[J]$ is contained in $E(\Kbar)_{\tors}$: indeed, if
$x\in T$ then there is $I\in J$ such that $Ix=0$. Since $R$ is an
order in a number field there is some non-zero integer $n\in I$, so
$nx=0$ and $x$ is torsion.

\begin{proposition}
	The $R$-module $E(\Kbar)[J]$ is $J$-injective.
\end{proposition}
\begin{proof}
	By \cite[Proposition 5.1]{lenstra1996complex}
	the $R$-module $E(\Kbar)_{\tors}$ is injective, thus in particular
	$J$-injective. Since $E(\Kbar)[J]=\divm{J}{0}{E(\Kbar)_{\tors}}$ it
	follows from Lemma \ref{lemma:InjectiveSub} that
	$E(\Kbar)[J]$ is $J$-injective.
\end{proof}

\begin{remark}
	Although not necessary for our applications, it is interesting to notice
	that in this setting $\Gamma$ is a maximal \jtextension{} of $M$.
	Indeed $E(\Kbar)/E(\Kbar)_{\tors}$ is a torsion-free module over
	the commutative integral domain $R$, so it is injective.
	Then the short exact sequence of $R$-modules
	\begin{align*}
		0\to E(\Kbar)_{\tors} \to E(\Kbar) \to E(\Kbar)/E(\Kbar)_{\tors}\to 0
	\end{align*}
	splits, so that $E(\Kbar)\cong E(\Kbar)/T\oplus T$ as $R$-modules and
	since $R$ is Noetherian it follows that $E(\Kbar)$ is injective.
	As in the above Proposition we may conclude that $\Gamma$ is $J$-injective,
	thus it is a maximal \jtextension{} of $M$.
\end{remark}

We now specialize to the case $J=\infty$.
\begin{remark} \label{remark:inftyIsN}
	Notice that in case $J=\infty$ we have $T=G(\Kbar)_{\tors}$ and
	\[
		\Gamma =
		\set{x\in E(\Kbar)\mid nx\in M\text{ for some } n\in\Z_{>0}}\,.
	\]
	If $R=\Z$ then $\End_R(T)$ is isomorphic, after fixing an isomorphism
	$T\cong (\Q/\Z)^2$, to $\Mat_{2\times 2}(\hat\Z)$. If $R$ is instead an
	order
	in an imaginary quadratic field then $\End_R(T)\cong R\otimes_\Z\hat\Z$.
	Indeed, fix for every prime $p$
	a $\Z_p$-basis for $R_p:=R\otimes_\Z\Z_p$ and consider the
	$\hat\Z$-subalgebra $C=\prod_pC_p$ of
	$\Mat_{2\times2}(\hat \Z)=\prod_p\Mat_{2\times2}(\Z_p)$,
	where $C_p$ is the image of the embedding of $R_p$ into
	$\Mat_{2\times 2}(\Z_p)$ given by its multiplication action on
	the $\Z_p$-module $\Z_p^2\cong R_p$. Then $R\otimes_\Z\hat\Z\cong C$ is
	a $\hat\Z$-algebra free of rank $2$ as a $\hat\Z$-module, since
	every $C_p$ is a $\Z_p$-algebra of rank $2$. Then for a suitable
	choice of an isomorphism $T\cong (\Q/\Z)^2$ we have
	\begin{align*}
		\End_R(T)&=\set{\varphi\in \End_\Z(T)\mid f(r(t))=r(f(t))\,
			\forall r\in R,\,t\in T}\\
			&=\set{\varphi\in\Mat_{2\times 2}(\hat\Z)\mid fc=cf\,
			\forall c\in C}\\
			&= C
	\end{align*}
	where the last equality follows by applying the Centralizer Theorem
	to the central simple $\Q_p$-subalgebra $R\otimes_\Z\Q_p$ of
	$\Mat_{2\times2}(\Q_p)$ and then restricting the coefficients to $\Z_p$.

	In both cases, the map $\tau$ coincides with the usual Galois representation
	associated with the torsion of $E$.
\end{remark}

\begin{proposition}
	Assume that the abelian group structure of $E(K)$ is known and that
	$M$ is given in terms of set of generators for $E(K)$.
	Then there exists an effectively computable positive integer $d$ such that
	\[ d\cdot\divm{\infty}{\tp(M)}{\tp(G(K))} \subseteq \tp(M)\,. \]
\end{proposition}
\begin{proof}
	First of all notice that $\tp(M)=M+T$ and $\tp(G(K))=G(K)+T$ seen
	as subgroups of $E(\Kbar)$. We conclude thanks to
	the considerations of \cite[\S6.1]{mypaper}.
\end{proof}

\begin{proposition}
	There exists an effectively computable positive integer $n$ such that
	\[ n\cdot H^1(\im(\tau),T) = 0\,. \]
\end{proposition}
\begin{proof}
	This follows from \cite[Proposition 6.3]{mypaper} and
	\cite[Corollary 6.8]{mypaper} in the non-CM case and from
	\cite[Proposition 6.12]{mypaper} in the CM case.
\end{proof}

\begin{proposition}
	There exists an effectively computable positive integer $m$ such that
	the subring of $\End_R(T)$ generated by $\im(\tau)$
	contains $m\cdot \End_R(T)$.
\end{proposition}
\begin{proof}
	This follows again from \cite[Corollary 6.8]{mypaper} in the case
	$R=\Z$ and from \cite[Theorem 1.5]{lombardo2017galois} in
	the CM case.
\end{proof}

\begin{theorem}
	\label{thm:Main}
	Assume that the abelian group structures of $E(K)$ and $M$ are 
	effectively computable.
	Then there exists an effectively computable positive constant $c$ such that
	the index of $\im(\kappa)$ in $\Hom(\Gamma/\tp(M),T)$ divides c.
\end{theorem}
\begin{proof}
	This is a direct consequence of Theorem \ref{thm:main} and the three
	propositions above.
\end{proof}

\begin{remark}
	\label{remark:av}
	Since Theorem \ref{thm:main} is stated in a fairly general form, one might
	wonder if it can be applied to obtain a version of Theorem \ref{thm:Main}
	for higer-dimensional abelian varieties.

	Provided that one is in, or can reduce to, a case in which $T$ is a
	$J$-injective $R$-module (for example if the abelian variety is simple
	and its endomorphism ring is a maximal order in a division algebra,
	see Remark \ref{remark:maximalorders}), the key steps are finding effective
	bounds for the integers $n$ and $m$ of Theorem \ref{thm:main}.
	Effective bounds for $m$ are known, see for example
	\cite[Théorème 1.5(2)]{remond}.

	It is also known (see \cite{lt}) that a bound for $n$ can be obtained
	by finding explicit homotheties in $\im(\tau)$. This seems a harder problem
	to tackle, but one can hope to reduce to finding homotheties in the
	images of the $\ell$-adic representations, as done in \cite[\S 7]{lt}.
	Explicit results on the existence of homotheties in the image of
	$\ell$-adic representations attached to abelian varieties are
	obtained for example in \cite{galateau}.
\end{remark}

\bibliographystyle{acm}
\bibliography{biblio}

\end{document}